\documentclass[12pt]{amsart}
\usepackage{amscd,amsmath,amsthm,amssymb}
\usepackage{pstcol,pst-plot,pst-3d}
\usepackage{lmodern,pst-node}
\usepackage{multicol}
\usepackage{epic,eepic}
\usepackage{amsfonts,amssymb,amscd,amsmath,enumerate,verbatim}
\psset{unit=0.7cm,linewidth=0.8pt,arrowsize=2.5pt 4}

\newpsstyle{fatline}{linewidth=1.5pt}
\newpsstyle{fyp}{fillstyle=solid,fillcolor=verylight}
\definecolor{verylight}{gray}{0.97}
\definecolor{light}{gray}{0.9}
\definecolor{medium}{gray}{0.85}
\definecolor{dark}{gray}{0.6}

\def\NZQ{\mathbb}               
\def\NN{{\NZQ N}}
\def\QQ{{\NZQ Q}}
\def\ZZ{{\NZQ Z}}

\def\FF{{\NZQ F}}
\def\GG{{\NZQ G}}

\def\KK{{\NZQ K}}
\def\frk{\mathfrak}               

\def\mm{{\frk m}}

\def\Phi{{\frk N}}

\def\eb{{\bold e}}
\def\ab{{\bold a}}

\def\xb{{\bold x}}

\def\al{\alpha}
\def\bl{\beta}

\def\opn#1#2{\def#1{\operatorname{#2}}} 
\opn\chara{char} \opn\length{\ell} \opn\pd{pd} \opn\rk{rk}
\opn\projdim{proj\,dim} \opn\injdim{inj\,dim} \opn\rank{rank}
\opn\depth{depth} \opn\grade{grade} \opn\height{height}
\opn\size{size}
\opn\embdim{emb\,dim} \opn\codim{codim}

\opn\Tr{Tr} \opn\bigrank{big\,rank}
\opn\superheight{superheight}\opn\lcm{lcm}
\opn\trdeg{tr\,deg}
\opn\reg{reg} \opn\lreg{lreg} \opn\ini{in} \opn\lpd{lpd}
\opn\size{size}\opn{\mult}{mult}
\opn{\Cl}{Cl}
\opn{\PF}{PF}
\opn{\RF}{RF}
\opn{\MC}{MC}
\opn{\Kos}{Kos}
\opn\div{div} \opn\Div{Div} \opn\cl{cl} \opn\Cl{Cl}
\opn\Spec{Spec} \opn\Supp{Supp} \opn\supp{supp} \opn\Sing{Sing}
\opn\Ass{Ass} \opn\Min{Min} \opn\cl{cl}
\opn\Ann{Ann} \opn\Rad{Rad} \opn\Soc{Soc}
\opn\Syz{Syz} \opn\Im{Im} \opn\Ker{Ker} \opn\Coker{Coker}
\opn\Am{Am} \opn\Hom{Hom} \opn\Tor{Tor} \opn\Ext{Ext}
\opn\End{End} \opn\Aut{Aut} \opn\id{id} \opn\ini{in}\opn\GCD{GCD}
\opn\nat{nat}
\opn\pff{pf}
\opn\Pf{Pf} \opn\GL{GL} \opn\SL{SL} \opn\mod{mod} \opn\ord{ord}
\opn\Gin{Gin}
\opn\Hilb{Hilb}\opn\adeg{adeg}\opn\std{std}\opn\ip{infpt}
\opn\Pol{Pol}
\opn\sat{sat}
\opn\Var{Var}
\opn\Gen{Gen}
\opn\lex{lex}
\opn\div{div}
\opn\NUF{NUF}
\opn\mNUF{mNUF}
\opn\type{type}

\opn\PF{PF}
\opn\Fr{F}
\opn\Ap{Ap}
\opn\aff{aff} \opn\con{conv} \opn\relint{relint} \opn\st{st}
\opn\lk{lk} \opn\cn{cn} \opn\core{core} \opn\vol{vol}
\opn\link{link} \opn\star{star}
\opn\gr{gr}

\def\pot#1#2{#1[\kern-0.28ex[#2]\kern-0.28ex]}

\opn\dirlim{\underrightarrow{\lim}}
\opn\inivlim{\underleftarrow{\lim}}
\let\union=\cup

\let\dirsum=\oplus

\let\iso=\cong

\let\Dirsum=\bigoplus

\let\to=\rightarrow

\def\Implies{\ifmmode\Longrightarrow \else
        \unskip${}\Longrightarrow{}$\ignorespaces\fi}
\def\implies{\ifmmode\Rightarrow \else
        \unskip${}\Rightarrow{}$\ignorespaces\fi}
\def\iff{\ifmmode\Longleftrightarrow \else
        \unskip${}\Longleftrightarrow{}$\ignorespaces\fi}

\let\:=\colon
\newtheorem{Theorem}{Theorem}[section]
\newtheorem{Lemma}[Theorem]{Lemma}
\newtheorem{lem}[Theorem]{Lemma}
\newtheorem{Corollary}[Theorem]{Corollary}
\newtheorem{Proposition}[Theorem]{Proposition}
\newtheorem{prop}[Theorem]{Proposition}
\newtheorem{Remark}[Theorem]{Remark}
\newtheorem{rem}[Theorem]{Remark}

\newtheorem{Example}[Theorem]{Example}
\newtheorem{ex}[Theorem]{Example}
\newtheorem{Examples}[Theorem]{Examples}
\newtheorem{Definition}[Theorem]{Definition}
\newtheorem{defn}[Theorem]{Definition}

\newtheorem{Question}[Theorem]{Question}
\let\epsilon\varepsilon
\let\kappa=\varkappa
\def\qed{\ifhmode\textqed\fi
      \ifmmode\ifinner\quad\qedsymbol\else\dispqed\fi\fi}
\def\textqed{\unskip\nobreak\penalty50
       \hskip2em\hbox{}\nobreak\hfil\qedsymbol
       \parfillskip=0pt \finalhyphendemerits=0}
\def\dispqed{\rlap{\qquad\qedsymbol}}

\opn\dis{dis}
\def\pnt{{\raise0.5mm\hbox{\large\bf.}}}

\opn\Lex{Lex}
\opn\int{int}

\newcommand{\inD}[1][\relax]{\def\argone{#1}\def\temprelax{\relax}
  \ifx\argone\temprelax\right.\else\,\middle|#1\right.{}\fi}

\newif\ifbinary
\binarytrue
\newcommand{\Z}{\mathbb Z}
\def\bld#1{\mbox{\boldmath $#1$}}

\begin{document}

\title{Almost symmetric numerical  semigroups}
\author{J\" urgen Herzog and Kei-ichi Watanabe}

\title{Almost symmetric numerical  semigroups}

\address{J\"urgen Herzog, Fachbereich Mathematik, Universit\"at Duisburg-Essen, Campus Essen, 45117
Essen, Germany} \email{juergen.herzog@uni-essen.de}

\address{Kei-ichi Watanabe, Department  of  Mathematics,  College of  Humanities  and Sciences, Nihon University,  Setagaya-ku, Tokyo,  156-8550}
\email{watanabe@math.chs.nihon-u.ac.jp}

\dedicatory{Dedicated to Professor Ernst Kunz}

\begin{abstract}
We study almost symmetric numerical semigroups and semigroup rings.
We  describe a characteristic property of  the minimal free resolution of the semigroup ring
of  an almost symmetric numerical semigroup.
 For almost symmetric semigroups generated by  $4$ elements  we will give a structure theorem by using the \lq\lq row-factorization matrices", introduced by Moscariello.
As a result, we give a simpler proof of Komeda's structure theorem of pseudo-symmetric numerical
semigroups generated by $4$ elements. Row-factorization matrices are also used to study shifted families of numerical semigroups.
\end{abstract}


\maketitle

\section*{Introduction}

Numerical semigroups and semigroup rings are very important objects in the study of singularities of
dimension $1$ (See \S 1 for the definitions).  Let $H =\langle n_1,\ldots ,n_e\rangle$ be a numerical
semigroup minimally generated by $\{n_1,\ldots,n_e\}$ and let $K[H] = K[ t^{n_1}, \ldots , t^{n_e}]$
be the semigroup ring of $H$, where $t$ is a variable and $K$ is any field.
We can represent $K[H]$ as a quotient ring of a polynomial ring $S=K[x_1,\ldots , x_e]$
 as $K[H] = S/I_{H}$, where $I_H$ is the kernel of the $K$-algebra homomorphism which maps
  $x_i\to t^n_i$.  We call $I_H$ the defining ideal of $K[H]$. The ideal $I_H$ is a binomial ideal, whose binomials correspond to   pairs of factorizations of elements of $H$.

\medskip 	
Among the numerical semigroups, almost symmetric semigroups admit very
interesting properties. They are a natural generalization of symmetric numerical semigroups. It was Kunz \cite{Ku} who observed  that $H$ is symmetric, if and only if is $K[H]$ Gorenstein. Almost symmetric semigroups are distinguished  by the symmetry of their pseudo-Frobenius numbers - a fact which has been  discovered by Nari \cite{N}.  The history of this class of numerical semigroups begins with the work of Barucci, Dobbs and Fontana \cite{BDF},  and the influential paper of Barucci and Fr\"oberg \cite{BF}. In \cite{BDF} pseudo-symmetric numerical semigroups appeared the first time, while in \cite{BF} almost symmetric numerical semigroups were introduced. The pseudo-symmetric are just a special class of almost symmetric numerical semigroups, namely those of type 2, see Section~\ref{sec:1} for details. Actually, Barucci and Fr\"oberg introduced more generally  the so-called  almost Gorenstein rings, which in the case of numerical semigroup rings lead to the concept of almost symmetric numerical semigroups.  Later on, Goto, Takahashi and Taniguchi \cite{GTT}  developed the theory of almost Gorenstein rings much further,  and extended this concept to rings of any Krull dimension.

The aim of this paper
is to analyze the structure of an almost symmetric numerical semigroup and its semigroup ring.

\medskip
In the first part of this paper, we analyze the structure of an almost symmetric semigroup by
using Apery sets and pseudo-Frobenius numbers. When our semigroup $H$ is generated by $3$
elements, then a characterization of $H$ to be almost symmetric is  known  in terms of the relation matrix of $I_H$, see  \cite{NNW}.
So in this paper, our main objective is to understand the structure of almost symmetric semigroups
generated by $4$ elements.   For our approach, the
RF-matrix $\RF(f)$ (row-factorization matrix) attached to a pseudo-Frobenius number $f$, as  introduced by
A. Moscariello in \cite{Mo}, is of  particular importance. It  will be shown that if $H$ is almost symmetric,
then $\RF(f)$ has \lq\lq special rows", as  described in  Corollary~\ref{0_in_row}. This concept of  special rows of  RF-matrices plays an essential role in our studies of almost symmetric semigroups generated by $4$ elements.

\medskip

One of the applications  of RF-matrices will be  the classification of pseudo-symmetric numerical semigroups generated by $4$ elements.
This result was already found by Komeda \cite{Ko},  but using RF-matrices, the argument becomes
much simpler. Also Moscariello (\cite{Mo}) proved that if $H$ is an almost symmetric numerical
 semigroup generated
by $4$ elements, then the type of $H$ is at most $3$. We give a new proof of this by using
the special rows of RF-matrices.  We also  show that in this case the defining relation of the semigroup
ring of $H$ is given by RF-matrices. This is not the case for arbitrary numerical semigroups.

\medskip

In the second part of this paper, we observe the very peculiar structure of the minimal free resolutions
of  $k[H]=S/I_H$ over $S$ when $H$ is almost symmetric. By this observation, we can see that if $e=4$,
and $\type(H)=t$, then the defining ideal needs at least $3(t-1)$ generators
 and also if $I_H$ is generated by exactly  $3(t-1)$  elements we can assert that the
 degree of each minimal generator of $I_H$ is of
the form $f + n_i+n_j$ where $f$ is a pseudo-Frobenius number different from the Frobenius number
of $H$.  Furthermore,  we show   that if $e=4$ and $H$ is almost symmetric, then
 $I_H$ is minimally generated by
either $6$ or $7$ elements and in the latter case one has $n_1+n_4=n_2+n_3$,  if we assume
$n_1<n_2<n_3<n_4$.

\medskip

In the last section we consider shifted families of numerical semigroups and study periodic properties  of $H$ under this shifting operation when $e=4$. Namely,
if $H=\langle n_1,\ldots ,n_4\rangle$, we put  $H+m=\langle n_1+m,\ldots ,n_4+m\rangle$ and ask
 when $H+m$ is almost symmetric for infinitely many $m$.  We prove that for any $H$,
 $H+m$ is almost symmetric of type $2$ for only finitely many $m$.  We also  classify
 those numerical semigroups  $H$ for which $H+m$ is almost symmetric of type $3$ for infinitely many $m$.

\medskip

Some readers will notice that there is a considerable overlap with  K. Eto's paper \cite{E}.  To explain
this, let us briefly comment on the history of the paper.  This work began in September  2015, when the 2nd
named author visited Essen. So it took several years for this work to be fully ripe. In the meantime we
gave  reports   in \cite{HW16}, \cite{HW17} in which we gave partial results  of contents of this paper
without proof.  Also, there were some occasions
that Eto and Watanabe discussed about this kind of problems and the use of RF-matrics.
Then Eto, by his approaches, which are  different from ours, more quickly completed the new proof of Theorem \ref{type_le3} and the proofs of some strucure theorems on almost symmetric numerical semigroups generated by 4 elements. Nevertheless we believe  that our point of view, which also includes the concept of RF-relations,  and our way of proof of the theorems may be useful for the future study of numerical semigroups.

\section{Basic concepts}
\label{sec:1}

In this section we fix notation and recall the basic
 definitions and concepts which will be used in this paper.

\medskip
\noindent
{\em  Pseudo-Frobenius numbers and Apery sets.} A  submonoid
$H\subset \NN$ with $0\in H$  and $\NN \setminus H$ is finite
 is called a numerical semigroup. Any numerical semigroup $H$ induces a partial order on $\ZZ$, namely $a\leq_H b$ if and only if $b-a\in H$.

There exist finitely many positive  integers $n_1,\ldots,n_e$ belonging to $H$ such that each $h\in H$ can be written as $h=\sum_{i=1}^e\al_in_i$ with non-negative integers $\al_i$. Such a presentation of $h$ is called a {\em a factorization} of $h$,  and  the set $\{n_1,\ldots,n_e\}\subset H$   is called a {\em set of generators} of $H$. If $\{n_1,\ldots,n_e\}$ is a set of generators of $H$, then we write $H=\langle n_1,\ldots ,n_e\rangle$. The set of generators $\{n_1,\ldots,n_e\}$ is called a {\em minimal set of generators} of $H$, if  none of the $n_i$ can be omitted to generate $H$. A minimal set of generators of $H$ is uniquely determined.

\medskip
Now let $H=\langle n_1,\ldots ,n_e\rangle$ be a  numerical semigroup. We assume  that $n_1,\ldots ,n_e$ are minimal generators of $H$, that  $\gcd(n_1,\ldots,n_e)=1$ and that $H\neq \NN$, unless otherwise stated.

The assumptions imply that the set $G(H)=\NN\setminus H$ of {\em gaps} is a finite non-empty  set. Its cardinality will be denoted by $g(H)$. The largest gap  is called the {\em Frobenius number} of $H$,  and denoted $\Fr(H)$.

An element $f\in \ZZ\setminus H$ is called a  {\em pseudo-Frobenius} number,
if $f+n_i\in H$ for all $i$. Of course, the Frobenius number is a  pseudo-Frobenius number as well and each pseudo-Frobenius number belongs to $G(H)$.
The set of  pseudo-Frobenius numbers will be  denoted by $\PF(H)$.

We also set $\PF'(H)=\PF(H)\setminus\{\Fr(H)\}$. The cardinality of $\PF(H)$ is called the {\em type} of $H$, denoted $t(H)$. Note that for any $a\in \ZZ\setminus H$, there exists $f\in \PF(H)$ such that $f-a\in H$.

\medskip\noindent
Let $a\in H$. Then we let
$$\Ap(a, H) = \{ h\in H\;|\; h-a \not\in H\}.$$
This  set is called the {\em Apery set} of $a$ in $H$. It is clear that $|\Ap(a,H)|=a$ and that $0$ and all $n_i$ belong to $\Ap(a,H)$.
For  every $a$  the largest element in $\Ap(a,H)$ is $a + \Fr(H)$.

\medskip\noindent
{\em Symmetric, pseudo-symmetric and almost symmetric numerical semigroups}. For each $h\in H$, the element $\Fr(H)-h$ does not belong to $H$. Thus the assignment $h\mapsto \Fr(H)-h$ maps each element $h\in H$ with $h<\Fr(H)$ to a gap of $H$. If each gap of $H$ is of the form $\Fr(H)-h$, then $H$ is called {\em symmetric}. This is the case if and only if for each $a\in \ZZ$ one has:  $a\in H$ \iff $\Fr(H)-a\not\in H$. It follows that a numerical  semigroup is symmetric if and only if  $g(G)=|\{h\in H\: \;   h<\Fr(H)\}|$, equivalently if $2g(H)=F(H)+1$.  A symmetric semigroup is also characterized by the property that its type is $1$. Thus we see that  a symmetric  semigroup satisfies  $2g(H)=\Fr(H)+t(H)$, while in general  $2g(H)\geq \Fr(H)+t(H)$. If equality holds, then  $H$  is called {\em almost symmetric}. The almost symmetric semigroups of type $2$ are called {\em pseudo-symmetric}. It is quite obvious that a numerical semigroup is pseudo-symmetric if and only if $\PF(H)= \{\Fr(H)/2, \Fr(H)\}$. From this  one easily deduces that if $H$ is pseudo-symmetric, then  $a\in H$ \iff $\Fr(H)-a\not\in H$ and  $a\neq \Fr(H)/2$.

Less obvious is the following nice result of Nari \cite{N} which provides a certain symmetry property of the pseudo-Frobenius numbers of $H$.

\begin{Lemma}
\label{Nari}
 Let $\PF(H)=\{f_1,f_2,\ldots,f_{t-1}, \Fr(H)\}$ with $f_1<f_2\ldots <f_{t-1}$. Then $H$ is almost symmetric if and only if
\[
f_i+f_{t-i}=\Fr(H) \quad \text{for} \quad i=1,\ldots,t.
\]
\end{Lemma}

\medskip
\noindent
{\em  Numerical semigroup rings}. Many of the properties of a numerical semigroup ring are reflected by algebraic properties of the associated semigroup ring. Let $H$ be a numerical semigroup, minimally generated by $n_1,\ldots,n_e$. We fix a field $K$. The semigroup ring $K[H]$ attached to $H$ is the $K$-subalgebra of the polynomial ring $K[t]$ which  is generated by the monomials $t^{n_i}$. In other words, $K[H]=K[t^{n_1},\ldots,t^{n_e}]$. Note that $K[H]$ is a $1$-dimensional Cohen-Macaulay domain. The symmetry of $H$ has a nice algebraic counterpart, as shown by Kunz \cite{Ku}. He has shown that $H$ is symmetric if and only if and only if $K[H]$ is a Gorenstein ring. Recall that a positively graded  Cohen--Macaulay $K$-algebra  $R$ of dimension $d$ with graded maximal ideal $\mm$ is Gorenstein if and only if $\dim_K\Ext_R(R/\mm, R)=1$. In general the $K$-dimension of the finite dimensional $K$-vector space is called the CM-type (Cohen-Macaulay type) of $R$.  Kunz's theorem follows from the fact that the type of $H$ coincides with the CM-type of $K[H]$.

If $a\in H$, then $K[H]/(t^a)$ is a $0$-dimensional $K$-algebra with $K$-basis $t^h+(t^a)$ with $h\in \Ap(a,H)$. The elements $t^{f+a}+(t^a)$ with $f\in PF(H)$ form a $K$-basis of the socle of  $K[H]/(t^a)$. This shows that indeed the type of $H$ coincides with the CM-type of $K[H]$.

\medskip
The canonical module 
of $\omega_{K[H]}$  of $K[H]$ can be identified with the fractionary ideal of $K[H]$ generated by the elements $t^{-f}\in Q(K[H])$ with $f\in \PF(H)$. Consider the exact sequence of graded $K[H]$-modules
\[
0\to K[H]\to \omega_{K[H]}(-\Fr(H))\to C\to 0,
\]
where $K[H]\to \omega_{K[H]}(-\Fr(H))$ is the $K[H]$-module homomorphism which sends $1$ to $t^{-\Fr(H)}$ and where $C$ is the cokernel of this map. One immediately verifies that $H$ is almost symmetric if and only if $\mm C=0$, where $\mm$ denotes the graded maximal ideal of $K[H]$. Motivated by this observation
Goto et al 
\cite{GTT} call a Cohen--Macaulay local ring with canonical module $\omega_R$ {\em almost Gorenstein}, if the exists an exact sequence
\[
0\to R\to \omega_R\to C \to 0.
\]
with $C$ an Ulrich module.  If $\dim C=0$, $C$ is a Ulrich module if and only if $\mm C=0$.
Thus it can be seen that $H$ is almost symmetric if and only if $K[H]$ is almost Gorenstein (in the graded sense).

\medskip
In this paper we are interested in the defining relations of $K[H]$. Let $S=K[x_1,\ldots,x_e]$ be the polynomial ring over $K$ in the indeterminates $x_1,\ldots,x_e$. Let $\pi: S\to K[H]$ be the surjective $K$-algebra homomorphism with $\pi(x_i)=t^{n_i}$ for $i=1,\ldots,n$.  We denote by $I_H$ the kernel of $\pi$. If we assign to each $x_i$ the degree $n_i$, then with respect to this grading, $I_H$ is a homogeneous ideal, generated by binomials. A binomial $\phi=\prod_{i=1}^ex_i^{\al_i}-\prod_{i=1}^ex_i^{\bl_i}$ belongs to $I_H$ if and only if $\sum_{i=1}^e\al_in_i=\sum_{i=1}^e\bl_in_i$. With respect to this grading $\deg \phi=\sum_{i=1}^e\al_in_i$.

\section{On unique factorization of elements of $H$ and Factorizations of $f+n_k$ for $f \in \PF(H)$}
\label{sec:2}
In this section we discuss unique factorization of elements of $H$ with respect to its minimal generator.
Also, we review Komeda's argument  on Apery set.

The following Lemma will be very essential in \S 4 and \S 5.

\begin{Lemma}
\label{Apery} Let $a\in H$ and $h\in\Ap(a,H)$. Then the following holds:
\begin{enumerate}
\item[{\em (i)}] If $h,h'\in H$ and if $h+h'\in \Ap(a,H)$, then $h,h'\in \Ap(a,H)$.
\item[{\em (ii)}] Assume $H$ is almost symmetric. If $h\in \Ap(a,H)$, then either
\[
(a+\Fr(H))-h\in \Ap(a,H)\quad \text{or}  \quad h - a\in \PF'(H).
\]
In the latter case, $(a+\Fr(H)) - h \in \PF(H)$.
\end{enumerate}
\end{Lemma}
\begin{proof}
(i) is obvious.  (ii) If $(a+\Fr(H))-h\not\in H$, there is some $h'\in H$ and $f\in \PF(H)$
such that $f = (a+\Fr(H))-h + h'$.  Since $h-a \not\in H$, $f\ne \Fr(H)$.
Then by Lemma ~\ref{Nari},
$f' = \Fr(H)-f = \Fr(H) - [(a+\Fr(H))-h + h'] = (h-a) -h'\in \PF'(H)$.
Since $h-a\not\in H$ by assumption,
we must have $h'=0$ and $h-a \in \PF'(H)$.
\end{proof}

Let  $H=\langle n_1,\ldots , n_e\rangle$ be a numerical semigroup minimally generated by $e$ elements.
For every $i$, $1\le i\le e$, we define $\al_i$ to be the minimal positive integer such that
\begin{eqnarray}
\label{minimal}
\al_i n_i = \sum_{j=1, j\ne i}^e \al_{ij} n_j.
\end{eqnarray}
Note that $\al_{ij}$ is in general  not  uniquely determined. However the minimality of $\al_i$ implies

\begin{Lemma}\label{al_iAp} For all integers $i$ and $k$ with  $1\le i,k\le e$ and  $i\ne k$ one has  $$(\al_i -1) n_i \in \Ap(n_k, H). $$
\end{Lemma}

\begin{proof} Suppose this is not the case. Then  $(\al_i -1) n_i - n_k\in H$,  and we will have an equation of type
$\beta_i n_i = \sum_{j=1, j\ne i}^e \beta_{ij} n_j$ with integers $0<\beta_i< \al_i$ and $\beta_{ij}\geq 0$, contradicting the
minimality of $\al_i$.
\end{proof}

Let $H$ be a numerical semigroup minimally generated by $n_1,\ldots,n_e$. An element $h\in H$ is said to have UF (Unique Factorization),  if $h$ admits only one factorization. Note that $h$ does not have UF if and only if $h\ge_H \deg(\phi)$ for some $\phi \in I_H$.

The set   $I=\{h\in H\: \; h  \text{ does not have UF} \}$ is an ideal of $H$ and is equal to $\{ \deg(\phi) \:\;  \phi \in I_H\}$.
Observe that  if $\phi\in I_H$ and $\deg(\phi)$ is a minimal generator of $U$, then $\phi$ is a minimal generator of
$I_H$.  But the converse is not true in general.  Hence the number of minimal generators of $I$ is less than or equal to the number of minimal generators of $I_H$.

\begin{Lemma}\label{deg_phi}
Let $\phi = m_1 - m_2$ be a minimal binomial generator of $I_H$.  Then the following holds:
\begin{enumerate}
\item[{\em (i)}] There exists $f \in \PF(H)$ and  integers  $i\ne j$ such that
$\deg(\phi)\le_H f+ n_i+n_j$.
\item[{\em (ii)}] Let  $i$ and $j$ such that $x_i|m_1$ and $x_j|m_2$. Then $\deg \phi= f + n_i  + n_j $ for some $f\in \PF' (H)$  if and only if
$\Fr(H) + n_i+n_j -\deg(\phi) \not\in H$.
 \end{enumerate}
\end{Lemma}

\begin{proof}
(i) We choose $i$ and $j$ such that $x_i|m_1$ and $x_j|m_2$.

Suppose that $\deg \phi - n_i-n_j \in H$. Then there exists a monomial $m$ with $\deg m=h$. Thus,  $mx_j-m_1/x_i\in I_H$ and $mx_i-m_2/x_j\in I_H$, and so $\phi=x_i(mx_j-m_1/x_i)+x_j(mx_i-m_2/x_j)\in I_H$, contradicting the minimality of $\phi$.
It follows that $\deg \phi - n_i-n_j \not\in H$. Hence there exist $f\in \PF(H)$ such that $f-(\deg \phi - n_i-n_j)\in H$, and this  implies that $\deg \phi \le_H f+ n_i+n_j$.

(ii) It follows from the proof of (i) that  $\deg \phi - n_j \in \Ap(n_i, H)$.
Then by Lemma~\ref{Apery}(ii) we obtain that $\deg \phi - n_j = f + n_i$ for some $f\in \PF'(H)$ if and only if
$(\Fr(H) + n_i) -  (\deg \phi - n_j) \not\in H$.
\end{proof}

The factorizations of the elements $f+n_k$ for $f\in \PF(H)$ play an important role in the understanding of the structure of $H$.  We first prove

\begin{Lemma}\label{f + n_k = b_in_i}  Let $f\in \PF(H)$. With the notation of {\em (\ref{minimal})}  the following holds:
\begin{enumerate}
\item[{\em (i)}] If $f + n_k = \sum_{j\ne k}\beta_jn_j$ and if $\beta_i\ge \al_i$ for some
$i$, then $\al_{ik}=0$.
\item[{\em (ii)}] If $f + n_k = b_in_i$ for some $k \ne i$. Then
$b_i\ge \al_i -1$.
\item[{\em (iii)}] If $f+n_k\leq_H (\alpha_i - 1)n_i$ for some $k\ne i$,
 then  $ f+n_k = (\al_i-1)n_i$.
\end{enumerate}
\end{Lemma}

\begin{proof}  (i) By using equation  (\ref{minimal}) we can replace the summand $\beta_in_i$ on the right hand side of the equation in (i) by $\sum_{k\neq i}\al_{ik}n_k+(\beta_i-\al_i)n_i$. Thus if   $\al_{ik} >0$, then $f\in H$, which is a contradiction.

 (ii) Add $n_i$ to both sides of the equality in (ii). Then  we have
\[ (b_i+1)n_i = (f + n_i) + n_k = \sum_{j \ne i} c_j n_j+ n_k.\]
Since the right hand side does not contain $n_i$, we must have $b_i+1 \ge \al_i$.\par

(iii) By assumption,  $ (\alpha_i - 1)n_i = f + n_k +  h$ for some $h\in H$. Write $f + n_k =\sum_j\beta_jn_j$ and $h=\sum_j\gamma_jn_j$ with non-negative integers  $\beta_j$ and $\gamma_j$. Then we get $(\al_i-1-\beta_i-\gamma_i)n_i=\sum_{j\neq i}(\beta_j+\gamma_j)n_j$. 
The minimality of $\al_i$ implies
$\al_i-1-\beta_i-\gamma_i=0$ and $\beta_j= 0$ for  $j\neq i$.
The second equations imply that $f+n_k=\beta_in_i$, and the first equation implies that
$\beta_i\leq \al_i-1$. On the other hand, by (ii) we have $\beta_i\geq \alpha_i-1$.
\end{proof}

In the case that $H$ is almost symmetric we can say more about the factorization of $f+n_k$.

\begin{Lemma}
\label{UF-almost symmetric}
Suppose that $H= \langle n_1,n_2,\ldots,n_e\rangle$ is almost symmetric, and let  $f \in \PF'(H)$.
\begin{enumerate}
\item[{\em (i)}] Suppose that  $f+n_k= \sum_{j\neq k}\beta_jn_j$ with  $\beta_i>0$,  Then there exists a factorization  $\Fr(H)+n_k=\sum_{j\neq k}a_jn_j$ of $\Fr(H)+n_k$ such
 that $\beta_i=a_i+1$ and $a_j\geq \beta_j$ for $j\neq i$.
\item[{\em (ii)}] Suppose that $\Fr(H)+n_k$ has UF, say $\Fr(H)+n_k=\sum_{j\neq k}a_jn_j$.  Then $f+n_k=(a_i+1)n_i$ for some $i\neq k$.
\end{enumerate}
\end{Lemma}

\begin{proof}
 (i) Put $h = (f+n_k) - n_i$. Since $\beta_i>0$, $h\in H$.
Then we see
\[\Fr(H)+ n_k =  h +  ((\Fr(H) - f) + n_i). \]
Since $H$ is almost symmetric,   $(\Fr(H) - f) \in \PF'(H)$ and
$(\Fr(H) - f) + n_i$ does not contain $n_i$ in its factorization.

(ii) Let   $f+n_k= \sum_{j\neq k}\beta_jn_j$ with  $\beta_i>0$. Assume there exists  $l\neq i$ with  $\beta_l>0$. Then (i) implies that
$\beta_l=a_l+1$. On the other hand, since $\beta_i>0$ we also have that  $a_j\geq \beta_j$ for all $j\neq i$. In particular, $a_l\geq \beta_l$, a contradiction. Thus $f+n_k=(a_i+1)n_i$.
\end{proof}

\begin{Corollary}
Suppose that $H= \langle n_1,n_2,\ldots,n_e\rangle$ is almost symmetric, and let  $f \in \PF'(H)$. If for some $k$, $f+n_k$ has a factorization with more than one non-zero coefficient, then $\Fr(H)+n_k$ does not have  UF .
\end{Corollary}

\begin{Lemma} Assume that $\Fr(H) + n_k$ has UF. Then the following holds:
\begin{enumerate}
\item[{\em (i)}] If
$\Fr(H) + n_k-(\al_j -1)n_j\in H$ for every $j\ne k$, then
$\Fr(H) + n_k = \sum_{j\ne k} (\al_j-1)n_j$.
\item[{\em (ii)}] If moreover, $H$ is almost symmetric, then $f +n_k=\al_in_i$ for
some $i \ne k$ and $f\in \PF'(H)$.
\end{enumerate}
\end{Lemma}

\begin{proof}
(i) Let $\Fr(H) + n_k =\sum_{j\neq k}a_jn_j$ be the unique factorization of $\Fr(H) + n_k$. Then the hypothesis in (i) implies that $a_j\geq \alpha_j-1$ for all $j\neq k$. If $a_i\geq \alpha_i$ for some $i\neq k$, then $\sum_{j\neq k}a_jn_j$ can be rewritten by using (\ref{minimal}), contradicting the assumption that $\Fr(H) + n_k$ has unique factorization.

(ii) follows from Lemma~\ref{UF-almost symmetric}(ii).
\end{proof}

The next two lemmata deal with the case that $e=4$, a case we are mainly interested in.

\begin{Lemma}\label{NU_e=4}  Let $e(H)=4$ and let $f\in \PF(H)$. Then the following holds:
\begin{enumerate}
\item[{\em (i)}] If $f+n_k$ does not have UF, then for some $i$,
$f + n_k\ge_H \al_in_i$.
\item[{\em  (ii)}] If $ f+n_k=\sum_{i\ne k} a_in_i$ with
$a_i< \al_i$ for every $i$ and for any factorization of $f+n_k$, then $f + n_k$ has UF.
\end{enumerate}
\end{Lemma}

\begin{proof}  (i) Let  $f +n_k =\sum_{j\neq k}\beta_jn_j$  and $f +n_k =\sum_{j\neq k}\beta_j'n_j$ be two distinct factorizations of $f+n_k$.  Then  $\prod_j x_j^{\beta_j}- \prod_j x_j^{\beta_j'}$ is a non-zero binomial  of $I_H$. Taking out the common factor of the difference, the remaining binomial $\phi=m_1-m_2$ belongs to $I_H$, since $I_H$ is a prime ideal,  and we have  $f +n_k\geq_H \deg \phi$. Moreover, $\phi$ does not contain the variable $x_k$. Thus
$\phi$ contains at most $3$ variables and since $gcd(m_1,m_2)=1$, $\phi$ must contain a monomial $x_i^b$ with $b\ge \alpha_i$. It proves that $f +n_k\geq_H bn_i$.

(ii) follows  from (i).
\end{proof}

\begin{Examples}{\em
(1) Let $H=\langle 22,28,47,53\rangle$. Then  $\PF(H)=\{25,258,283\}$. Since $25+258=283$,
$H$ is almost symmetric  of type $3$. Moreover,  $I_H=(xw-yz, xy^3-w^2, x^2y^2-zw, x^3-z^2, x^{13}z- y^{10}w,
x^{14} - y^{11})$ , so that $\al_1= 14, \al_2=11, \al_3=2=\al_4$.

In this example, $f + n_i$ has UF
for every $i$ and $f\in \PF'(H)$. In these cases, condition (ii) of Lemma~\ref{NU_e=4} is satisfied. On the other hand, each of $283+47$ and $283+53$ has 3 factorizations.

\medskip
(2) Let $H = \langle 33, 56, 61,84\rangle$. Then $\PF(H) = \{28, 835, 863\}$. Since $28+835=863$,
$H$ is almost symmetric  of type $3$. Moreover,  $I_H= (xw-yz, x^{27}z-y^2w^{10},  y^3-w^2, xy^2-zw, x^2y-z^2,
x^{28}-w^{11})$, so that $\al_1= 28, \al_2=3, \al_3=2=\al_4$. In this example, $835 + n_1$ and
$835+n_3$ do not have UF.  For example,  $835 + n_3$ has 6 different factorizations, including
$835 + n_3= 16 n_2$. }
\end{Examples}

\begin{Lemma}\label{formula_Ap}  Assume that  $H$ is almost symmetric. Then the following holds:
\begin{enumerate}
\item[{\em (i)}]  We assume $e=4$. If for some $k$, $\al_{ik} \ge 1$ for all $i\ne k$, then $\Fr(H) + n_k$ has UF.

\item[{\em (ii)}] Assume that $\Fr(H) +n_k =
\sum_{j \ne k} \beta_jn_j$.  If $\Fr(H) +n_k$ has UF,
then $n_k = \prod_{j\ne k} (\beta_j+1) + \type(H) -1$.
\end{enumerate}
\end{Lemma}

\begin{proof} (i) If $\Fr(H) + n_k$ does not have UF, Lemma~\ref{NU_e=4} implies that there exists a factorization  $\Fr(H) + n_k= \sum_{i\ne k} \beta_in_i$ with  $\beta_i \ge \al_i$ for some $i$. But since we assume that $\al_{ik}\ge 1$, this yields that  $\Fr(H)\in H$, a  contradiction.

(ii)  We first observe that
\[
\Ap(n_k, H)=\{h\in H\: \; h\leq_H F(H)+n_k\}\union\{f+n_k\:\; f\in PF'(H)\}.
\]
Thus,  $n_k=\type(H)-1+|\{h\in H\: \; h\leq_H F(H)+n_k\}|$.

Since by assumption $F(H)+n_k$ has UF, each $h\leq_H F(H)+n_k$ has UF, as well. Therefore, each  $h\leq_H \Fr(H)+n_k$ has a unique  factorization $h=\sum_{i\ne k} \gamma_in_i$ with integers $0\leq \gamma_i\leq \beta_i$, and conversely each such sum $\sum_{i\ne k} \gamma_in_i$ is $\leq_H\Fr(H)+n_k$. It follows that that
$|\{h\in H\: \; h\leq_H F(H)+n_k\}| =\prod_{i\neq k}(\beta_i+1)$, as desired.
\end{proof}

\section{$\RF$-matrices}
\label{sec:3}
Let us recall the notion of the row-factorization matrix ($\RF$-matrix for short) introduced by Moscariello in \cite{Mo} for the numerical semigroup $H=\langle n_1,\ldots,n_e\rangle$. It describes for each $f\in\PF(H)$ and each $n_i$ a factorization of $f+n_i$.

\begin{Definition}\label{RF-def}
Let $f\in \PF(H)$.  An $e\times e$ matrix $A= (a_{ij})$ is an {\em $RF$-matrix} of $f$,
if $a_{ii}= -1$ for every $i$, $a_{ij}\in \NN$ if $i\ne j$
 and for every $i= 1,\ldots , e$,
$$\sum_{j=1}^e a_{ij} n_j = f.$$
\end{Definition}

Note that an $\RF$-matrix of $f$ need not to be uniquely determined. Nevertheless,  $\RF(f)$ will be the notation for one of the possible $\RF$-matrices of $f$.

\subsection{Fundamental Properties}

The most important property of $\RF$-matrices is the following.

\begin{Lemma}\label{symm}
 Let $f,f' \in \PF(H)$ with $f + f'\not\in H$. Set $\RF(f)= A=(a_{ij})$
and $\RF(f') = B = (b_{ij})$. Then either $a_{ij}= 0$ or $b_{ji} =0$ for every pair
$i\ne j$. In particular, if $\RF(\Fr(H)/2) = (a_{ij})$, then either $a_{ij}$
or $a_{ji}=0$ for every $i\ne j$.
\end{Lemma}

\begin{proof} By our assumption,  $f + n_i =\sum_{k\ne i} a_{ik} n_k$ and
$f' + n_j = \sum_{l \ne j} b_{jl} n_l$.  If $a_{ij}\ge 1$ and $b_{ji}\ge 1$, then
summing up these equations, we get
\[  f + f'= (b_{ji} - 1)n_i + (a_{ij}-1)n_j + \sum_{s\ne i,j} (a_{is} + b_{js})n_s
\in H,\]
a contradiction.
\end{proof}


\begin{Remark}\label{symm-Rem}
{\em As mentioned before, for given $f\in\PF(H)$, $\RF(f)$ is not necessarily unique.
Note that in the notation of Lemma \ref{symm}, if $a_{ij}>0$ for {\bf some}
$\RF(f)$, then $b_{ji} = 0$ for {\bf any} $\RF(f')$.}
\end{Remark}

The rows of $\RF(f)$ produce binomials in $I_H$. We shall need the following notation. For a vector $\ab=(a_1,\ldots,a_n) \in \ZZ^n$, we let $\ab^+$ the vector  whose $i$th entry is $a_i$ if $a_i\geq 0$, and is zero otherwise, and we let $\ab^-=\ab^+-\ab$. Then $\ab=\ab^+-\ab^-$ with $\ab^+,\ab^-\in \NN^n$.

\begin{Lemma}
\label{produce}
Let $\ab_1,\ldots,\ab_e$ be the row vectors of $\RF(f)$,
and set $\ab_{ij}=\ab_i-\ab_j$ for all $i,j$ with $1\leq i<j\leq e$. Then
$
\phi_{ij}=\xb^{\ab_{ij}^+}-\xb^{\ab_{ij}^-}\in I_H
$
for all $i<j$. Moreover, $\deg \phi_{ij} \leq f+n_i+n_j$. Equality holds, if the vectors $\ab_i+\eb_i+\eb_j$ and $\ab_j+\eb_i+\eb_j$ have disjoint support, in which case there is no cancelation when taking the difference of these two vectors.
\end{Lemma}

\begin{proof}
Let $\eb_1,\ldots,\eb_e$ be  the canonical unit vectors of $\ZZ^e$. The   vector $\ab_i+\eb_i+\eb_j$ as well as the vector $\ab_j+\eb_i+\eb_j$ with $i<j$  is the coefficient vector of a factorization of  $f+n_i+n_j$. Thus, taking the difference of these two vectors,  we obtain the vector $\ab_{ij}$ with $\sum_{k=1}^ec_kn_k=0$ where the $c_k$ are the components of $\ab_{ij}$.  It follows that $\phi_{ij}=\xb^{\ab_{ij}^+}-\xb^{\ab_{ij}^-}$ belongs to $I_H$.  The remaining statements are obvious.
\end{proof}

{
We call a  binomial relation of $\phi\in I_H$ an {\em $\RF(f)$-relation},  if it is of the form as in Lemma~\ref{produce}, and we call it an {\em $\RF$-relation} if it is an  $\RF(f)$-relation for some  $f\in \PF'(H)$.

\begin{Example}\label{RF_to_I_H}{\em  Let $H=\langle 7,12,13, 22 \rangle$.
Then  $\PF(H)= \{15,30\}$. In this case,
$$\RF(15) =  \left( \begin{array}{cccc} -1 & 0 & 0 & 1\\2 & -1 & 1 & 0\\
4 & 0 & -1 & 0 \\ 0 & 2 & 1 & -1
 \end{array}\right).$$
Taking the difference  of the first and second row  we get the vector $(-3,1,-1,1)$. This gives us the  minimal generator $yw-x^3z$ of $I_H$, where we put
$(x_1,x_2,x_3,x_4)=(x,y,z,w)$. \par

\medskip
Let us now consider all the choices of 2 rows and the resulting generators of $I_H$ together with their degree.
\begin{eqnarray*}
\ab_{12} &=& (-3,1,-1,1),  \quad\hspace{1cm} yw-x^3z,
\quad \hspace{0.57cm} 34\\
\ab_{13} &=& (-5, 0,1,1), \quad  \hspace{1.3cm} zw-x^5,  \quad
 \hspace{0.83cm} 35      \\
\ab_{14} &=& (-1,-2,-1,2), \quad \hspace{0.65cm} w^2-xy^2z,  \quad \hspace{0.45cm}44        \\
\ab_{23} &=& (-2,-1,2,0),  \quad \hspace{0.96cm} z^2-x^2y, \hspace{0.75cm}
  \quad 26       \\
\ab_{24} &=& (2,-3,0,1),  \quad \hspace{1.3cm}  x^2w- y^3,  \quad
 \hspace{0.63cm}36  \\
\ab_{34} &=& (4,-2,-2,1), \hspace{1.4cm} x^4w-y^2z^2, \quad  \hspace{0.25cm}50
\end{eqnarray*}

Among the above elements of $I_H$,
 $ x^4w-y^2z^2 = x^2(x^2w- y^3) - y^2( z^2-x^2y)$ is not a
minimal generator of $I_H$,  and we obtain $5$ minimal generators of $I_H$. Cancelation occurs in computing $\ab_{23}$ and $\ab_{24}$. }
\end{Example}}

\begin{Question}
\label{question} {\em  Are all minimal minimal generators of a numerical  semigroups  $\RF$-relations?}
\end{Question}

The following lemma provides a condition which guarantees  that $I_H$ can be generated by $\RF$-relations.

\begin{Lemma}
\label{conditionforquestion}
Suppose $I_H$ admits a system of binomial generators $\phi_1,\ldots,\phi_m$ satisfying the following property:
 for each $k$ there exist $i<j$ and $f\in \PF'(H)$  such that $\deg \phi_k=f+n_i+n_j$ and $\phi_k=u-v$, where $u$ and $v$ are monomials such that $x_i|u$ and $x_j|v$. Then $I_H$ is generated by $\RF$-relations.
 \end{Lemma}

\begin{proof}
Let $\phi_k=u-v$ be as described in the lemma, and let $u'=u/x_i$.  Then $\deg u'=f+n_j$ and $\deg u=(f+n_j)+n_i$ . Let $v'$ be a monomial with $\deg v'=f+n_i$. Then $\deg x_j v'=(f+n_i)+n_j$. By construction,   $\psi_k=u-x_j v'$ is an $\RF$-relation, and we have
\[
\phi_k=\psi_k+x_j(v/x_j-v') \quad\text{with}\quad v/x_j-v'\in I_H.
\]
It follows that $(\phi_1,\ldots,\phi_m) +\mm I_H=(\psi_1,\ldots,\psi_m) +\mm I_H$. Nakayama's lemma implies that $I_H=(\psi_1,\ldots,\psi_m)$.
\end{proof}

It will be shown in Theorem~\ref{type2} that  Question~\ref{question}  has an affirmative answer when $e=4$ and $\Fr(H)/2\in \PF(H)$. Here we show

\begin{Proposition}
\label{3generated}
Question~\ref{question} has an affirmative  answer if $e=3$
\end{Proposition}

\begin{proof}

Let $H=\langle n_1,n_2,n_3\rangle$ be a 3-generated  numerical semigroup. We first consider the case that $H$ is not symmetric and collect a few known facts.

By Herzog \cite{H} and Numata \cite[Section 2.2]{Nu}
the following facts are known: there exist positive integers $\alpha$, $\beta$ and $\gamma$, as well  $\alpha'$, $\beta'$ and $\gamma'$  such that
\begin{enumerate}
\item[(1)] $I_H$ is minimally generated by $g_1=x^{\alpha+\alpha'}-y^{\beta'}z^\gamma$, $g_2=y^{\beta+\beta'}-x^\alpha z^{\gamma'}$ and $g_3=z^{\gamma+\gamma'}-y^\beta z^{\alpha'}$, that is, $(\alpha+\alpha')n_1=\beta' n_2+\gamma n_3$, $(\beta+\beta')n_2=\alpha n_1+\gamma' n_3$ and  $(\gamma+\gamma')n_3=\alpha' n_1+\beta n_2$.
\item[(2)]   $n_1=(\beta+\beta')\gamma +\beta'\gamma'$, $n_2=(\gamma+\gamma')\alpha+\gamma'\alpha'$ and $n_3=(\alpha+\alpha')\beta+\alpha'\beta'$.
\item[(3)] $\PF(H)=\{f,f'\}$ with $f=\alpha n_1+(\gamma-\gamma')n_3-(n_1+n_2+n_3)$ and $f'= \beta' n_2+(\gamma-\gamma')n_3-(n_1+n_2+n_3)$.
\end{enumerate}

\medskip
Then it is easy to see that.
\[
\RF(f) =  \left( \begin{array}{ccc} -1 & \beta'+\beta-1 & \gamma-1\\
\alpha-1 & -1 & \gamma+\gamma'-1\\
\alpha +\alpha'-1 & \beta-1 & -1
\end{array}\right).
\]
Let $\ab_1,\ab_2, \ab_3$, the first, second and third row of $\RF(f)$. Then we obtain
$\ab_3-\ab_1=(\alpha+\alpha', -\beta',-\gamma)$, $\ab_1-\ab_2=(-\alpha, \beta+\beta', -\gamma')$ and $\ab_2-\ab_3=(-\alpha',-\beta, \gamma+\gamma')$.

Comparing these vectors with (1), we see that they correspond to the minimal  generators of $I_H$.

\medskip
We now consider the case  that $H$ is a symmetric semigroup. Then it is known that there exist positive integers $a,b$  and $d$ with $\gcd(a,b)=1$ and $d>1$ such that (after a relabeling) $n_1=da$,$n_2=db$ and $n_3=\alpha a-\beta b$. In this case $I_H$ is generated by the regular sequence $g_1=x^b-y^a$ and $g_2= z^d-x^{\alpha}y^{\beta}$, that is, for the semigroup we have the generating relations $bn_1=an_2$ and $dn_3=\alpha n_1+\beta_3 n_2$.

Since $H$ is symmetric $\PF(H)=\{\Fr(H)\}$, and since $I_H$ is  generated by the regular sequence $F(H)=(\deg g_1)+(\deg g_2)-n_1-n_2-n_3= bn_1+gn_3-n_1-n_2-n_3$,

\medskip
We  claim that

\[
\RF(F(H)) =  \left( \begin{array}{ccc} -1 & a-1 & d-1\\
b-1 & -1 & d-1\\
b-1+\alpha_1 & \alpha_2-1 & -1
\end{array}\right),
\]
\medskip
\noindent
if $\alpha_2>0$. On the other hand, if $\alpha_2=0$ and $\alpha_1>0$, then the last column of $\RF(F(H))$ has to be replaced by the column  $(\alpha_1-1,a-1+\alpha_2,-1)$.

 Let $\ab_1,\ab_2, \ab_3$, the first, second and third row of $\RF(f)$. Then, if $\alpha_2>0$ we obtain
 $\ab_1-\ab_2=(b,-a,0)$ and $\ab_3-\ab_2=(\alpha_1,\alpha_2, -d)$. These are the generating relations of $H$. In the case that $\alpha_1>0$ one obtains $\ab_1-\ab_3=(\alpha_1,\alpha_2, -d)$.
\end{proof}

We will show that if  $e=4$ and $H$ is pseudo-symmetric or almost symmetric,
 then Question~\ref{question} has an affirmative answer
  ( cf.  Theorem \ref{Type_even} and Theorem
 \ref{rfrelations} and Remark \ref{RF-rel-7gen}.)  When $e=4$ and $H$ is symmetric, we can not determine
$\RF(\Fr(H))$ in a unique manner and also we cannot get all the generators of $I_H$ from a
single $\RF(\Fr(H))$ although we get the generators by selecting a suitable expression and
consider linear relations of rows of $\RF(\Fr(H))$.

Moscariello proves \cite[Lemma 6]{Mo} that if  $H$ is almost symmetric and  $e=4$,  and if for some $j$ and $a_{ij}=0$ for every $i\ne j$, then
$f = \Fr(H)/2$.  His result can be slightly improved.

\begin{Lemma}\label{a_ij>0} Assume  that $H$ is almost symmetric and  $e=4$. Let $f\in \PF(H), f\ne \Fr(H)$ and put $A=(a_{ij})=\RF(f)$.
Then for every $j$, there exists $i\neq j$ such that $a_{ij}>0$. Namely, any column of $A$ should
contain some positive component. 
\end{Lemma}

\begin{proof}
First, let us recall the proof of Moscariello.
Assume, for simplicity, $a_{i1}=0$ for $i=2,3,4$.  Put $d = \gcd(n_2,n_3,n_4)$.
From the equation $f = -n_2 + a_{23}n_3 + a_{24}n_4$, we have $d|f$ and from
$f = -n_1+a_{12}n_2+a_{13}n_3+a_{14}n_4$, we get $d| n_1$. This implies $d=1$.
Hence $H_1: =\langle n_2,n_3,n_4\rangle$ is a numerical semigroup,  and the last 3 columns of $A$  show us that
$f\in \PF(H_1)$.

Since $H_1$ is generated by $3$ elements, $t(H_1)\le 2$ (see \cite{H}), and
since $\Fr(H_1)\ge \Fr(H)>f$, we conclude  that $\PF(H_1)=\{f, \Fr(H_1)\}$.
Hence we have $\Fr(H_1) -f \not\in H_1$ and $\Fr(H_1) -f \le_{H_1} f$ or $\Fr(H_1) -f \le_{H_1} \Fr(H_1)$. The second case cannot happen, since otherwise $f\in H_1$. Thus we have $\Fr(H_1) -f \le_{H_1} f$.

On the other hand,  since $\Fr(H)\not\in H_1$ and $f< \Fr(H)\le \Fr(H_1)$, we get  $\Fr(H) - f \le_{H_1} \Fr(H_1) - f \le_{H_1} f$. Moreover,  since
$ \Fr(H) - f \in \PF(H)$, it follows that  $ \Fr(H) - f =f$.
(Until here, this  Moscariello's argument.)

The arguments before  show that $F(H)=F(H_1)$,   and that  $H_1$ is  pseudo-symmetric. The latter implies that  $g(H_1) = \Fr(H_1)/2 +1$.Therefore, $g(H) \le g(H_1)-1 =\Fr(H_1)/2=\Fr(H)/2$, since
$H \supset H_1\cup\{n_1\}$.  This  is a contradiction.
\end{proof}

Combining Lemma \ref{symm}, Remark \ref{symm-Rem} and Lemma \ref{a_ij>0}, we get

\begin{Corollary}\label{0_in_row}
Assume $H$ is almost symmetric, and let $f\in \PF'(H)$. Then every row of $\RF(f)$ has at least one $0$.
Moreover, for every $i$, there exists $j\ne i$ such that the $(i,j)$ component of $\RF(f)$ is $0$ for  any choice of $\RF(f)$.
\end{Corollary}

If $e=5$, Lemma \ref{a_ij>0} is not true.

\begin{ex}  If $H= \langle 10, 11 , 15, 16, 28 \rangle$, then $\PF(H) =\{ 5,17,29, 34\}$
and hence $H$ is almost symmetric of type $4$. Then
\[\RF(5) =  \left( \begin{array}{ccccc} -1 & 0 & 1 & 0 & 0\\
                                                            0 & -1 & 0 & 1 & 0\\
2 & 0 & -1 & 0 & 0\\ 1  & 1 & 0 &  -1& 0\\ 0 & 3 & 0& 0 & -1
 \end{array}\right), \quad {\mbox{\rm and}}  \;
 \RF(29) =  \left( \begin{array}{ccccc} -1 & 1 & 0 & 0& 1\\1 & -1 & 2 & 0 & 0\\
2 & 2& -1 & 0 & 0\\ 0 & 0 & 3  & -1 & 0\\ 3 & 1 & 0 & 1 & -1
 \end{array}\right).\]
We see that the 5th column of $\RF(5)$ has no positive entry and we can choose another
expression of $\RF(29))$ whose  $(5,3)$ entry is positive.
\end{ex}

We can say that if $e=4$ and $H$ almost symmetric, then $f + n_i$ has UF in $H$ in most cases
for $f\in \PF'(H)$.

\begin{Corollary}\label{2nj}  Assume that  $H$ is almost symmetric, $e=4$. Then  for any  $f\in PF'(H)$ and
any $n_i$, the decomposition of $f+n_i$ has at most $2$ $n_j$'s.  Moreover, if $f+n_i$ does not have
UF and $n_j, n_k$ appears in the decomposition of $f+n_i$, then we have $\al_jn_j= \al_kn_k$.
\end{Corollary}
\begin{proof}
Let $\{i,j,k,l\}$ be a permutation of $\{1,2,3,4\}$.
By Corollary \ref{0_in_row}, we may assume that the $(i.l)$ component of $\RF(f)$ is $0$
 for any choice of $\RF(f)$.  Thus $f+n_i$ contains only $n_j, n_k$. Assume that there
 are $2$ different expressions $f+n_i = an_j + bn_k =  a'n_j + b'n_k$. Assuming $a>a', b< b'$,
  we have $(a-a')n_j = (b'-b)n_k$ and then $a\ge a-a'\ge \al_j,  b'-b \ge \al_k$.
\end{proof}
\begin{Proposition}\label{f + n_k = bn_i}
Suppose that  $H$ is almost symmetric, $e=4$ and for some  $f\in PF'(H)$ we have  $f + n_k = bn_i$ for some $k \ne i$.
Then one of the following cases occur:
\begin{enumerate}
\item[{\em (i)}] $b= \al_i -1$  or
\item[{\em (ii)}] $b \ge \al_i$ and for some $j\ne i,k$, $\al_i n_i = \al_j n_j$.
\end{enumerate}
\end{Proposition}
\begin{proof}   This is a direct consequence of
 Lemma~\ref{f + n_k = b_in_i}  Corollary \ref{2nj}.
\end{proof}

\begin{Example}{\em  Let $H =\langle 33,56,61,84\rangle$ with $\PF(H)=\{f=28,f'=835, F(H)=863\}$
and $\al_1=28, \al_2=3, \al_3=2, \al_4=2$. In this case,  $\RF(28)$ is uniquely determined,
but there are several choices of $\RF(835)$. Among them we can choose the following,
where  $ f' + n_3= 16 n_2$ with $16> \al_2$. Note that
we have $\al_2n_2= \al_4n_4$ in this case.

$$\RF(28) =  \left( \begin{array}{cccc} -1 & 0 & 1 & 0\\0 & -1 & 0 & 1\\
1 & 1 & -1 & 0 \\ 0 & 2 & 0 & -1
 \end{array}\right)\quad \text{and}\quad
  \RF(835) =  \left( \begin{array}{cccc} -1 & 2 & 0 & 9\\27 & -1 & 0 & 0\\
0 & 16 & -1 & 0 \\ 26 & 0 & 1 & -1
 \end{array}\right).$$}
\end{Example}

The following Proposition plays an important role in Section~\ref{sec:komeda}.

\begin{Proposition}\label{f+n_k= f' + n_l}  Assume $e=4$, $H$ is almost symmetric  and $\{i,j,k,l\}=\{1,2,3,4\}$.
Then the following statements hold:.
\begin{enumerate}
\item[{\em (i)}] For any $f,f',f'' \in \PF'(H)$,  $f+n_k = f' + n_l=f''+n_j$ does not occur.

\item[{\em (ii)}] Assume that for some  $f\ne f'\in  \PF '(H)$, we have
$f+n_k = f' + n_l$ for some $1\le k,l\le 4$. Then $f+n_k = f' + n_l = (\al_i -1)n_i$
 for some $i\ne j,k$.
\item[{\em (iii)}] Assume  $f+n_k = f' + n_l = (\al_i -1)n_i$ for some $f,f'\in \PF'(H)$ with
$f+f' \not\in H$.  Then there are expressions
\[\al_in_i = pn_j + qn_k = p'n_j + r n_l \quad {\mbox{\text with }} q,r>0.\]
Moreover, if $p\le p'$, then $q = \al_k$ and if $q\ge \al_k+1$, then
$\al_kn_k=\al_jn_j$.

\end{enumerate}
\end{Proposition}

\begin{proof} (i) Assume for some $f,f',f'' \in \PF'(H), f+n_k = f' + n_l=f''+n_j$. Then these terms are equal to  $bn_i$ for some integer $b>0$.
Then, adding $\Fr(H)-f$, we have $\Fr(H)+ n_k = (\Fr(H)-f) +  n_l+ f'= (\Fr(H)-f)+n_j+f'' =
(\Fr(H)-f) +bn_i$.  This implies that $n_k$ does not appear in $(\Fr(H)-f)+n_l, (\Fr(H)-f)+n_j,
(\Fr(H)-f)+n_i$, contradicting Lemma \ref{a_ij>0}.  \par

(ii)   Assume that $f+n_k = f' + n_l = bn_i + b'n_j$. If $b,b'>0$,
adding $\Fr(H) -f$ to both sides, we get
\[ \Fr(H) +n_k = (\Fr(H) -f) + (f' + n_l) = (\Fr(H) -f) + bn_i + b'n_j,\]
where $\{i,j,k,l\}$ is a permutation of $\{1,2,3,4\}$.
 Then $n_k$ does not appear in $(\Fr(H) -f)+n_i, (\Fr(H) -f)+n_j$. Moreover,
 since $\Fr(H) +n_k = [(\Fr(H) -f) +n_l] + f'$, $n_k$ does not appear in
 $(\Fr(H) -f)+n_l$, too. This contradicts  Lemma \ref{a_ij>0}.

 Hence $b=0$ or $b'=0$. We may assume that $b'=0$, and hence  $f+n_k = f' + n_l = bn_i$.
 But if $b\ge \al_i$, then $\al_i n_i$ cannot contain $n_k, n_l, n_j$, which is absurd.
 Hence by Proposition \ref{f + n_k = bn_i}, we must have $b =\al_i-1$.\par

(iii) We write the $(i,j)$ component of $\RF(f)$ (resp. $\RF(f')$) by $m(i,j)$ (resp. $m'(i,j)$).
We know by Lemma \ref{symm} that $m(i,j) >0$ for some expression of $\RF(f)$, then
$m'(j,i)=0$ for any expression of $\RF(f')$ and vice versa.

Now, adding $n_i$ to both sides of the equation, we get
\[
(f+n_i) +n_k = (f' +n_i) + n_l = \al_i n_i.
\]

Hence  $\al_{ik}>0$ and $\al_{il}>0$ in some  expressions of
$\al_in_i$.  If we have an expression of the form
$\al_i n_i = p n_j + q n_k + rn_l$ with $q,r >0$, then we have
\[ f + n_i = p n_j + (q-1) n_k + rn_l \quad {\mbox{\text and }} f' + n_i = p' n_j +q n_k + (r-1)n_l,\]
\noindent
so $m(i,l), m'(i,k) >0$. This  contradicts Lemma \ref{symm} since $m(k,i), m'(l,i)>0$.
\par
Hence we have different expressions
\[\al_in_i = p n_j + q n_k = p' n_j + r n_l\]
with $q,r>0$.  If $p\le p'$, then we
have $qn_k = (p'-p) n_j + r n_l$. Hence we have $q\ge \al_k$.
\par
If $q\ge \al_k+1$, then $f+n_i$ has 2 different expressions and since
$p n_j + (q-1) n_k$ cannot contain $n_l$, we must have $\al_kn_k = \al_jn_j$
by Corollary \ref{2nj}.
\end{proof}

\begin{Example} {\em The following example demonstrates the result of Proposition~ \ref{f+n_k= f' + n_l} (ii).  Let $H= \langle 9,22,46,57 \rangle$. Then $\PF(H) = \{ f=35,f'=70,\Fr(H)=105\}$. We have  $f +n_4= 92 = f' + n_2 = 2 n_3$, and   $\al_1=10, \al_2=3,
\al_3=3, \al_4=2$. The $\RF$-matrices  are 

\[  \RF(35) =  \left( \begin{array}{cccc} -1 & 2 & 0 & 0\\0 & -1 & 0 & 1\\
9 & 0 & -1 & 0 \\ 0 & 0 & 2 & -1
 \end{array}\right), \quad {\mbox{\rm and}}  \;
 \RF(70) =  \left( \begin{array}{cccc} -1 & 1 & 0 & 1\\0 & -1 & 2 & 0\\
8 & 2& -1 & 0 \\ 9 & 0 & 1 & -1
 \end{array}\right).\]}
\end{Example}

\section{Komeda's structure theorem for $4$ generated pseudo-symmetric semigroups via
RF-matrices. }\label{sec:komeda}

We will apply our results in the previous sections to give a new proof
 of Komeda's structure theorem for  $4$-generated pseudo-symmetric
 semigroups using  $\RF(\Fr(H)/2)$.

 We  believe that  our proof of the structure theorem of type $2$ almost symmetric numerical semigroups is simpler than the one in the original paper \cite{Ko}.

\medskip
In this subsection we always assume that $H=\langle n_1,n_2,n_3,n_4\rangle$ and that
$\Fr(H)/2\in \PF(H)$.

First we sum up the properties of $\RF(\Fr(H)/2) = A=(a_{ij})$ given in Lemma \ref{symm},
\ref{a_ij>0} and Corollary \ref{0_in_row}.

\begin{Proposition}\label{RF_F/2}
Let  $\RF(\Fr(H)/2)=(a_{ij})$  be an $\RF$-matrix of $\Fr(H)/2)$. Then:

\begin{enumerate}
\item[{\em (i)}]  $a_{ii} = -1$ for every $i$,  and $a_{ij}$ is a non-negative integer for every $i\ne j$.
\item[{\em (ii)}] For every pair $(i,j)$ with $i\ne j$, either $a_{ij}$ or $a_{ji}$ is $0$.
\item[{\em (iii)}] Every row and column of $A$ has at least one positive entry.
\end{enumerate}
\end{Proposition}

 Since there are at most $6$ positive entries in $A$, at least $2$ rows have
 only one positive entry. More precisely we have

\begin{Proposition}\label{RF(F/2)}  After a suitable relabeling of the generators of $H$  we may assume that
 \[\RF( \Fr(H)/2) = \left( \begin{array}{cccc} -1 & \alpha_2-1 & 0 & 0\\
0 & -1 & \alpha_3-1 & 0 \\
a & 0 & -1 & d\\
a' & b & 0 & -1 \end{array}\right)\]
with $a', d>0$  and $a, b\ge 0$.
 \end{Proposition}

\begin{proof} We simply write $F=F(H)$ and let $A=(a_{ij})= \RF(F/2)$.
We will see that this matrix is uniquely determined, but at this moment,
we can take any choice.
 Proposition~\ref{RF_F/2}(ii) implies that $A$ has at least $6$ entries
  equal to zero. Thus there must exist a row of $A$ with only one positive entry
 and we can assume that the first row of $A$ is $(-1, b',0,0)$, or
 $F/2 + n_1 = b' n_2.$  We can assume that in {\it any} choice of $\RF( \Fr(H)/2)$,
 $(1,3), (1,4)$ entries are $0$.  Then, by Lemma \ref{f + n_k = b_in_i},
  $b' \ge \al_2 -1$. If $b' \ge \al_2$, we have another choice of $\RF(F/2)$
  with non $0$ $(1,3)$ or  $(1,4)$ entry. Hence we conclude $b' =\al_2-1$.

 Now,  by Lemma \ref{symm}, $a_{21}=0$ and
 $F/2+n_2= a_{23}n_3+a_{24}n_4$.  If $a_{23}$ or
 $a_{24}=0$, then we may assume the $2$nd row of $\RF(F/2)$ is
 $(0,-1,c,0)$.  If $a_{23}$ and   $a_{24}$ are both positive, then since
 $a_{32}=a_{42}=0$ and $a_{34}$ or $a_{43}=0$ by Lemma \ref{symm},
 the 3rd row (resp.\ 4th row) of $\RF(F/2)$ is $(a', 0,-1, 0)$ (resp.
 $(a', 0, 0,-1)$. In both cases, after changing the order of generators,
 we can assume  the first $2$ rows of $\RF(F/2)$ are

\[ \left( \begin{array}{cccc} -1 & \al_2-1 & 0 & 0\\
0 & -1 & c & 0  \end{array}\right).\]

or

\[ \left( \begin{array}{cccc} -1 & b & 0 & 0\\
0 & -1 & \al_3-1 & 0  \end{array}\right).\]

We assume the 1st expression. We can treat the 2nd case in the same manner.

We have only to show that $c=\al_3-1$ since
 $a_{32}=0$ and $a_{34}>0$ by \ref{symm}.
 Now we use $a,a',b,d$ as in the right hand side of $\RF(F/2)$ in Proposition
 \ref{RF(F/2)}.  If $c\ge \al_3$, $\al_3n_3$ can contain only $n_4$,
 since $a_{21}=0$.   Then we must have
 $\al_3n_3 = \al_4 n_4$ by Corollary \ref{2nj}.
 Then by Lemma \ref{symm}, $b=0$ and
 $F/2 + n_4 = a' n_1$. It is easy to see that $a'=\al_1-1$,
 since $a_{42}=a_{43}=0$ by our assumption.

 Taking the difference of 3rd and 4th rows of $\RF(F/2)$, we get

 \[(d+1) n_4 = (\al_1-1-a) n_1+ n_3,\]

 and thus $d\ge \al_4-1$. On the other hand, from
  $F/2 + n_3 = an_1+ d n_4$, we see $d\le  \al_4-1$ since we have
 seen  $\al_3n_3= \al_4n_4$.  Then we have

 \[(\al_3 -1) n_3 = (\al_1-1-a) n_1,\]

 contradicting the definition of $\al_3$.
\end{proof}

Now we come to the main results of this section.

\begin{Theorem}\label{type2} Let $H=\langle n_1,n_2,n_3,n_4\rangle$,  and assume
$\PF(H) =\{\Fr(H)/2, \Fr(H)\}$. Then for a suitable relabeling of the generators of $H$,
\begin{eqnarray}\label{normal}
 \RF(\Fr(H)/2) = \left( \begin{array}{cccc} -1 & \alpha_2-1 & 0 & 0\\
0 & -1 & \alpha_3-1 & 0 \\
\al_1-1 & 0 & -1 & \al_4-1\\
\al_1-1 &  \al_{42} & 0 & -1 \end{array}\right)
\end{eqnarray}
and  $\Fr(H)/2 +n_k$ has UF for every $k$, that is, $\RF(\Fr(H)/2)$ is uniquely determined.
\end{Theorem}


\medskip
\begin{proof}[Proof of Theorem \ref{type2}]
We start with the $\RF$-matrix $A$ given in  Proposition~\ref{RF(F/2)},  and first determine the unknown
values $a,a',b,d$ there.  Note that until the end of (3), we only assume that $\Fr(H)/2\in \PF(H)$.

\medskip

(1) Taking the difference of  the $1$st and $2$nd (resp.\ the $2$nd and $3$rd)
 row of $A$, leads to the equations
\[(*2) \quad \al_2n_2 = n_1 + (\al_3 -1)n_3,\]
\[(*3) \quad \al_3 n_3 = an_1 + n_2 + d n_4.\]

\medskip
(2) If $a'\ge \al_1$, since  (4, 3) component of $\RF(F/2)$ should be $0$
in any expression, we must have  $\al_1n_1 = \al_2n_2$. But then from
(1), we have another expression of $\RF(F/2)$ with positive (4,3) component.
A contradiction! Similarly, we have $b\le \al_2-1$.

Now, taking  the difference of the $1$st and  the $4$th row, we have
\[(\al_2-1- b)n_2+n_4= (a'+1)n_1.\]

Since we have seen $a'\le \al_1 -1$, we have $a'= \al_1-1$.
Also, since $n_4$ is a minimal generator of $h$, we  have $b < \al_2-1$. We have
\[(*1) \quad \al_1n_1 = (\al_2-1-b) n_2 + n_4\].
Then, we must have $a< \al_1$, since otherwise will have different expression of
$\RF(F/2)$ with positive $(3,2)$ entry.
\par

\medskip
(3) If $d \ge \al_4$, then we must have $\al_4n_4 = \al_1n_1$ since $(3,2)$ entry
of $\RF(F/2) =0$ in any expression. Then we have a contradiction from equation $(*1)$.

Taking  the difference of the $3$rd and  the $4$th row, we have
\[ (d+1)n_4 = (\al_1 - 1 -a)  n_1 + b n_2 + n_3.\]
Since $d\le \al_4-1$, we have $d= \al_4-1$ and we get
\[(*4) \quad \al_4n_4 = (\al_1 - 1 -a)  n_1 + b n_2 + n_3.\]
\medskip
Let us  sum up what we have got so far:

\[
(\dagger) \quad \RF(\Fr(H)/2) =  \left( \begin{array}{cccc} -1 & \alpha_2-1 & 0 & 0\\
0 & -1 & \al_3-1 & 0 \\
a & 0 & -1 & \al_4-1\\
\al_1-1 &
 b=\al_{42} & 0 & -1 \end{array}\right),
\]

\medskip
\noindent
Moreover, we have obtained an expression with
$\al_{12}=\al_2-1-b >0$, $\al_{13}=0$, $\al_{14}=1$,  $\al_{21}=1$, $\al_{23}=\al_3-1$, $\al_{24}=0,$
$\al_{31}=a$, $\al_{32}=1$, $\al_{34}=\al_4-1$, $\al_{41}=\al_1-1-a$, $\al_{42}=b$, $\al_{43}=1.$

\medskip
\noindent
(4) To finish the proof, it suffices to show that $b>0, a>0$ and then $a= \al_1-1$.
If $b=0$, then adding 2nd and 4th rows of our matrix, we have
\[(\al_1-1)n_1+ (\al_3-1)n_3 = \Fr(H) + n_2+ n_4.\]
Since $(\al_i -1)n_i \in \Ap(n_k, H)$ for every $k\ne i$, and from Lemma \ref{Apery},
we get
\[\Fr(H) +n_2 - (\al_1-1)n_1=  (\al_3-1)n_3 - n_4 \in \PF'(H).\]
Which leads to $ (\al_3-1)n_3 = n_4 + \Fr(H)/2 = (\al_1-1)n_1$, a contradiction!
Thus we have $b>0$. If $a=0$, then adding 2nd and 4th rows of our matrix and by the same
argument as above, we have $(\al_2-1)n_2= \Fr(H)/2 + n_3= (\al_4 -1)n_4$, a contradiction.

\medskip

Next we show that $a=\al_1-1$. We have seen that $a\le \al_1-1$.  If $a< \al_1-1$, then
 we have $\al_{41}>0$ and by Lemma \ref{formula_Ap},
$\Fr(H) + n_1$ has UF.
We  show that this leads to a contradiction.\par

Adding the 1st and 2nd rows of $\RF(\Fr(H)/2)$, we get
\[ \Fr(H) + n_1 = (\al_2-2) n_2 + (\al_3-1) n_3.\]
Since  $\Fr(H) + n_1$ has UF, $(\Fr(H) + n_1) - (\al_4-1)n_4\not\in H$ and by
Lemma \ref{Apery} we should have
 $(\al_4-1)n_4 = \Fr(H)/2 +n_1$.  Since we have seen $\Fr(H)/2 +n_1= (\al_2-1)n_2$,
 we get a contradiction.  Hence we have $a= \al_1-1$.
\end{proof}

\begin{Theorem}\label{Type_even}
  If $\Fr(H)/2 \in \PF(H)$ and if $\RF(\Fr(H)/2)$ is as in Theorem \ref{type2},
then
we have:
\begin{enumerate}
\item[{\em (i)}] $\Fr(H) +n_2$ has UF and we have $ n_2 = \al_1\al_4(\al_3-1) +1$.
\item[{\em (ii)}] Every generator of $I_H$ is an $\RF(\Fr(H)/2)$-relation.\\
Namely,
$I_H = (x_2^{\al_2}- x_1x_3^{\al_3-1},
x_1^{\al_1}- x_2^{\al_{2}-1-\al_{42}}x_4,
x_3^{\al_3}- x_1^{\al_1-1}x_2x_4^{\al_4-1},
x_3^{\al_3-1}x_4- x_1^{\al_1-1}x_2^{\al_{42}+1},
x_4^{\al_4} - x_2^{\al_{42}}x_3)$.
(The difference of 1st and 3rd rows does not give a
minimal generator of $I_H$.)
\item[{\em (iii)}] $H$ is almost symmetric and $\type(H)=2$.
\end{enumerate}
\end{Theorem}

We will show in Proposition \ref{even_type_le2}  that if $e=4$ and $\Fr(H)/2\in \PF(H)$, then
$\RF(\Fr(H)/2)$ is as in \ref{normal}, showing that if $e=4$ and $H$ is almost
symmetric with even $\Fr(H)$, then $\type(H)=2$.

\begin{proof}
First, note that  by Lemma \ref{formula_Ap}(i),  $\Fr(H) + n_2$ has UF,  since $\al_{i2}\neq 0$ for all
$i\neq 2$, as we can read from our $\RF(\Fr(H)/2)$.
Adding the 2nd and 3rd row of $\RF(\Fr(H)/2)$, we get
\[ \Fr(H) +  n_2 = (\al_1-1)n_1+ (\al_3-2)n_3 + (\al_4-1) n_4.\]
Hence $n_2 = \al_1\al_4(\al_3-1) +\type(H)-1\ge \al_1\al_4(\al_3-1) +1$,
by Lemma \ref{formula_Ap}(ii). We will show that $\type(H)=2$ by showing
$n_2 =  \al_1\al_4(\al_3-1) +1$. \par
\medskip

We determine the minimal generators of $I_H$. Let $I'$ be the ideal generated by the binomials
\[x_2^{\al_2}- x_1x_3^{\al_3-1},
x_1^{\al_1}- x_2^{\al_{12}}x_4,
x_3^{\al_3}- x_1^{\al_1-1}x_2x_4^{\al_4-1},
x_3^{\al_3-1}x_4- x_1^{\al_1-1}x_2^{\al_2-\al_{12}},
x_4^{\al_4} - x_2^{\al_2-1-\al_{12}}x_3.
\]
Since these binomials correspond to difference vectors of rows of $\RF(\Fr(H)/2)$, it is clear that $I' \subset  I_H$. In order to  prove that $I' = I_H$, we first show that  $S/ (I' ,x_2) = S/(I_H,x_2)$. Note that $S/(I_H,x_2)\iso K[H]/(t^{n_2})$. Therefore, $\dim_KS/(I_H,x_2)=n_2$, and since $I'\subset I_H$ we see that $\dim_K S/(I',x_2)\geq n_2$. We have seen that $n_2 = \al_1\al_4(\al_3-1) +\type(H)-1$. On the other hand,
\[
S/(I',x_2)\iso K[x_1,x_3,x_4]/( x_1x_3^{\al_3-1},  x_1^{\al_1}, x_3^{\al_3}, x_3^{\al_3-1}x_4, x_4^{\al_4}),
\]
from which we deduce that  $\dim_k (S/ (I' ,x_2))= \al_1\al_4(\al_3-1) +1$. It follows that $\al_1\al_4(\al_3-1) +1\geq  \al_1\al_4(\al_3-1) +\type(H)-1$. This is only possible if $\type(H)=2$ and $S/ (I' ,x_2) = S/(I_H,x_2)$.

Now consider the exact sequence
\[
0\to I_H/I'\to S/I'\to S/I_H\to 0.
\]
Tensorizing this sequence with $S/(x_2)$ we obtain the long exact sequence
\[
\cdots \to \Tor_1(S/I_H, S/(x_2))\to (I_H/I')/x_2(I_H/I')\to S/(I',x_2)\to S/(I_H,x_2)\to 0.
\]
Since $S/I_H$ is a domain, $x_2$ is a non-zerodivisor on $S/I_H$. Thus   $\Tor_1(S/I_H, S/(x_2))=0$. Hence, since  $S/ (I' ,x_2) = S/(I_H,x_2)$,  we deduce from this exact sequence that $(I_H/I')/x_2(I_H/I')=0$. By Nakayama's lemma, $I_H/I'=0$, as desired.

This finishes the proof of the structure theorem of Komeda, using $\RF(\Fr(H)/2)$.
\end{proof}
\begin{Remark}{\em  The generators of $I_H$ in the papers \cite{Ko} and  \cite{BFS} are obtained by applying the cyclic  permutation
$(1,3,4,2)$. Namely, in their paper, $I_H = (x_1^{\al_1}- x_3x_4^{\al_4-1},
x_1^{\al_1}- x_2^{\al_{2}-1-\al_{42}}x_4,
x_4^{\al_4}- x_3^{\al_3-1}x_1x_2^{\al_2-1},
x_4^{\al_4-1}x_2- x_3^{\al_3-1}x_1^{\al_{21}+1},
x_2^{\al_2} - x_1^{\al_{21}}x_4)$.

These equations are derived from the matrix obtained by the same permutation.

\[\RF(\Fr(H)/2)=  \left( \begin{array}{cccc} -1 &  0 & 0 &\alpha_4-1 \\
\al_{21} & -1 & \alpha_3-1 & 0 \\
\al_1-1 & 0 & -1 & 0\\
0 & \al_2-1 & \al_3-1 & -1 \end{array}\right) \]
}
\end{Remark}

\section{Some structure Theorem for $\RF$-matrices of a $4$-generated almost almost symmetric
 numerical semigroup $H$  and a proof that  $\type(H) \le 3$.}
 \label{sec:5}

We investigate the \lq\lq special rows" of  $\RF$-matrices of a $4$-generated almost almost symmetric
 numerical semigroup $H$ and give  another proof of the fact that a $4$-generated
almost symmetric  numerical semigroup has type $\le 3$, proved by A.~Moscariello.

\begin{Theorem}\label{type_le3}\cite{Mo}  If $H=\langle n_1,\ldots, n_4\rangle$ is almost symmetric, then $\type(H)\le 3$.

\end{Theorem}

In this section,  let $H =\langle n_1,n_2,n_3,n_4\rangle$  and
we assume always that $H$ is almost symmetric.  Our main tool is the \lq\lq special row of $\RF(f)$ for
 $f\in \PF(H)$.

 \begin{defn}\label{special_def}
 {\em If a row of $\RF(f)$ is of the form $(\al_i -1) {\bld e}_i - {\bld e}_k$ is called {\bf a special row},
 where ${\bld e}_i$ the $i$-th unit vector of $\Z^4$.}
 \end{defn}

\begin{lem}\label{special row}  We assume $e=4, \{n_1,n_2,n_3,n_4\} =
\{n_i, n_j, n_k, n_l\} $ and $H$ is almost symmetric.
 \begin{enumerate}
\item[{\em (i)}]  There  are $2$ rows in $\RF(\Fr(H)/2)$  of the form
$(\al_i -1) {\bld e}_i - {\bld e}_k$. 
\item[{\em (ii)}]  If $f\ne f' \in \PF(H)$ with $f+f' \not\in H$, then there are $4$ rows in
$\RF(f)$ and $\RF(f')$ of the form $(\al_i -1) {\bld e}_i - {\bld e}_k$.
\item[{\em (iii)}] For every pair $\{f,f'\}\subset \PF'(H), f\ne f', f+f'\not\in H$ and for
every $j\in \{1,2,3,4\}$,
 there  exists some $s$ such that either $(\al_j-1)n_j = f +n_s$ or
$(\al_j-1)n_j = f' +n_s$.
\end{enumerate}
\end{lem}
\begin{proof} We have proved (i) in Proposition \ref{RF_F/2}. \par
(ii) By Lemma \ref{symm}, we have at least $12$ zeroes in $\RF(f)$ and
$\RF(f')$.  Also, by
 Corollary \ref{0_in_row}, a row of $\RF(f)$ should not contain $3$
 positive components. Also, we showed in \ref{f + n_k = b_in_i}
that if the $k$-th row of $\RF(f)$ is $-{\bld e}_k + b{\bld e}_i$,
then $b\ge \al_i -1$.

 Now, a row of $\RF(f), \RF(f')$ is either one
 of the following $3$ types.
 \begin{enumerate}
\item[(a)] contains $2$ positive components,
\item[(b)] $q {\bld e}_s - {\bld e}_t$  with $q\ge \al_s$,
\item[(c)] $(\al_s-1){\bld e}_s - {\bld e}_t$.
\end{enumerate}
Now, in case (b),  some different components of $t$-th row
are positive. Hence if $a,b,c$ be number of rows or type (a), (b), (c),
respectively, we must have $2(a+b)+c\le 12$.  Since $a+b+c=8$,
we have $c\ge 4$. \par
 (iii) We have shown in (ii) that there are at least $4$ rows
 in $\RF(f)$ and $\RF(f')$ of the form $(\al_i -1) {\bld e}_i - {\bld e}_k$.
So, we may assume that for some  $i,k,l$ we have the relations
\[n_k + f = (\al_i-1)n_i = n_l+ f'.\]
Then  we have shown in Proposition \ref{f+n_k= f' + n_l} that
$\al_in_i$ must have $2$ different expressions
\[(5.7.1) \qquad \al_i n_i = p n_j + qn_k = p' n_j + r n_l\]
and we will have
\[(5.7.2) \qquad f + n_i = p n_j + (q-1)n_k, f'+ n_i = p' n_j + (r-1) n_l.\]

We will write $\RF(f) = (m_{st})$ and $\RF(f') = (m'_{st})$
if $f+ n_s = \sum_t m_{st} n_t$ for some expression and
we say $m_{st}=0$ for some $(s,t)$ if $m_{st} =0$ in {\it any}
expression of $f+ n_s = \sum_t m_{st} n_t$ and likewise for
$\RF(f') = (m'_{st})$.
\par
Here, since our argument is  symmetric on $k,l$ until now, we may assume that 
$p'\ge p$ and then from (5.7.1), we have
\[ qn_k = (p'-p)n_j + rn_l\]
and we must have $q \ge \al_k$ and also if
$q\ge \al_k + 1$, then $\al_kn_k=\al_jn_j$  by Proposition \ref{f + n_k = bn_i}.
 But then from $p n_j + qn_k
=(p+\al_j)n_j +(q-\al_k)n_k= p' n_j + r n_l$, we have $r\ge \al_l$.
 and this will easily lead to a contradiction. Hence we have $q=\al_k$ and
\[(5.7.2') \qquad f + n_i = p n_j + (\al_k-1)n_k, f'+ n_i = p' n_j + (r-1) n_l.\]
\par
Now, since $m_{il}=m_{kl}=0$ (resp. $m'_{ik} = m'_{lk}=0$),
then $m_{jl}$ (resp. $m'_{jk}$) must be positive by Lemma
\ref{a_ij>0}. Let us put
\[ f + n_j = s n_i + t n_k + u n_l \quad (u>0).\]

Since $f' = f + n_k - n_l$, we have $f' = s n_i + (t+1) n_k + (u-1)n_l$.

We discuss according to $s>0$ or $s=0$.\par

\medskip
Case (a).  If $s>0$, by Lemma \ref{a_ij>0}, $m_{jl}=m'_{jk}=0$
since  by Lemma \ref{symm}, $m_{ij}=m'_{ij}=0$ and hence
we have $p=p'=0$ and after  a short calculation we have
\[(5.7.3) \qquad f + n_i =  (\al_k-1)n_k, f' +n_i = (\al_l-1)n_l\]
by Lemma \ref{f + n_k = b_in_i} and since $m_{il}=m'{ik}=0$.
Also, by Corollary \ref{0_in_row}, since $s>0$, $m_{li}=m'_{ki}=0$
and by Lemma \ref{symm}, either $m_{lk}$ or $m'_{lk}=0$.
Hence we have $f + n_l = (\al_j-1)n_j$ (resp. $f' +n_k =  (\al_j-1)n_j$)
if $m_{lk}$ (resp. $m'_{lk}=0$). And then we have proved our assertion.

Case (b). $s=0$.  We put
\[(5.7. 4)\qquad f + n_l = an_i + bn_j + cn_k, f' + n_k =  b'n_j + d' n_l\]
(since $m_{ik} =\al_k-1>0$, we have $m'_{ki}=0$ by Lemma \ref{symm})
so that
\[RF(f)= \left( \begin{array}{cccc} -1 &  p & \al_k-1 & 0 \\
0 & -1 & t & u \\
\al_i-1 & 0 & -1 & 0\\
a & b & c & -1 \end{array}\right),\quad
\RF(f') =
\left( \begin{array}{cccc} -1 &  p' & 0 & r-1 \\
0 & -1 & t+1 & u-1 \\
0 & b' & -1 & d'\\
\al_i-1 &0 & 0 & -1 \end{array}\right).\]

Note that by Lemma \ref{symm}, we have
\[a(r-1) = b(u-1)= b't = cd' =0.\]

We then discuss the cases $c>0$ and $c-0$.

Case (b1). $c>0$.  Then  $d'=0, b'= \al_j -1$ and $t=0$.  Then $u = \al_l-1$ and
we need to show there is a special row $(\al_k -1) {\bld e}_ k - {\bld e}_s$ for some $s$.
In this case, by Corollary~\ref{0_in_row}, either $b=0$ or $a=0$. If $b>0$, then $u=1$ and
the 2nd row of $\RF(f')$ is $(\al_k -1) {\bld e}_ k - {\bld e}_j$ with $\al_k=2$. If $a>0$, then
the ist row of $\RF(f')$ is $p' {\bld e}_ j - {\bld e}_i$ and this contradicts the fact the 3rd
row of $\RF(f')$ is $(\al_j-1) {\bld e}_ j - {\bld e}_k$ inducing $n_i=n_k$. If $a=b=0$, then
the 4th row of $\RF(f)$ is the desired special row.  Now we are reduced to the case $c=0$.

Case (b2). $c=0$ and $b>0$.  Then we have $t+1= \al_k-1$. If $a>0$, then we have $r=1$
and $p'=\al_j-1$ and either $b'=0$ or $t=0$.  In either case, we have enough special rows.
If $a=0$, then $b=\al_j-1$ and the 4th row of $\RF(f)$ is $(\al_j-1) {\bld e}_ j - {\bld e}_l$.
If, moreover,  $b'=0$, then $d' = \al_l-1$ and the 3rd row of $\RF(f')$ is the desired special row.
If, moreover,  $b'>0$, then we have $t=0$ and the 2nd rows of $\RF(f)$ and $\RF(f')$ give
enough numbers of special rows.

\end{proof}

We investigate the semigroups which has special type of $\RF(f)$.

\begin{lem}\label{4rows} Assume $H$ is almost symmetric with odd $\Fr(H)$ and assume
for some $f\in \PF'(H)$, $\RF(f)$ has only one positive entry in each row. Then we have:
\begin{enumerate}
\item[{\em (i)}] After suitable permutation of indices, we can assume
\[\RF(f) = \left( \begin{array}{cccc} -1 & \alpha_2-1 & 0 & 0\\
0 & -1 & \al_3-1 & 0 \\
0 & 0 & -1 & \al_4-1\\
\al_1-1 & 0 & 0 & -1 \end{array}\right).\]
\item[{\em (ii)}] In this case, if we put $f' =\Fr(H)-f$, then
 \[\RF(f') = \left( \begin{array}{cccc} -1 & \alpha_2-2 & \al_3-1 & 0\\
0 & -1 & \al_3-2 & \al_4-1 \\
\al_1-1 & 0 & -1 & \al_4-2\\
\al_1-2 & \al_2-1 & 0 & -1 \end{array}\right)\]
\item[{\em (iii)}] We have  $\type(H) = 3$ with $\PF(H)=
\{f,f', \Fr(H)\}$, $\mu(I_H)=6$ and the minimal generators of $I_H$ are obtained
from taking differences of $2$ rows of $\RF(f)$, namely $I_H=(x_1^{\al_1}-x_2^{\al_2-1}x_4, x_2^{\al_2}-x_3^{\al_3-1}x_1, x_3^{\al_1}-x^{\al_4-1}x_2, x_4^{\al_4}-x_1^{\al_1-1}x_3,
x_1x_4^{\al_4-1}-x_2^{\al_2-1}x_3, x_1^{\al_1-1}x_2-x_3^{\al_3-1}x_4).$
\item[{\em (iv)}]
We have
\[n_1= (\al_2-1)(\al_3-1)\al_4 +\al_2, n_2= (\al_3-1)(\al_4-1)\al_1,\]
\[n_3= (\al_4-1)(\al_1-1)\al_2+\al_4, n_4=(\al_1-1)(\al_2-1)\al_3+\al_1,\]
\end{enumerate}
\end{lem}
\begin{proof}  Assume $(i, a_i)$ component of $\RF(f)$ is positive.  Then
$(a_1.a_2.a_3,a_4)$ gives a permutation of $(1,2,3,4)$ with no fixed point.
So, we can assume either $(a_1.a_2.a_3,a_4)= (2,1,4,3)$ or $(2,3,4,1)$.

Since $f + n_1= (\al_2-1)n_2$, we have
\[f + n_1+n_2 = \al_2 n_2 = (\al_3-1)n_3 + n_1,\]
hence $\al_{21}=1,\al_{23}=\al_3-1,\al_{24}= 0$ and likewise, we have
$\al_{12}= \al_2-1, \al_{13}=0, \al_{14}=1, \al_{31}=0,\al_{32}=1, \al_{34}=\al_4-1,
\al_{41}= \al_1-1, \al_{42}=0,\al_{43}=1$.

Next,  by Lemma \ref{symm}, $\RF(f')$ is of the form
$\RF(f') = \left( \begin{array}{cccc} -1 & p_2 & p_3 & 0\\
0 & -1 & q_3& q_4 \\
r_1 & 0 & -1 & r_4\\
s_1 & s_2 & 0 & -1 \end{array}\right).$  Then, note that we should have
$p_2<\al_2$ because $\al_{21}=1>0$ and $p_3<\al_4$ because $\al_{34}>0$
and likewise.

Them we compute $\Fr(H)+n_1= (\al_2-1)n_2+ f' = (\al_2-2)n_2 + q_3n_3+q_4n_4
= p_2n_2+p_3n_3+ f = (p_2-1)n_2 + (\al_3-1+p_3)n_3 = p_2n_2+(p_3-1)n_3 +
(\al_4-1)n_4$. Hnce we have $p_2=\al_2-2$ and $q_4= \al_4-1$. Repeating this process,
we get $\RF(f')$ as in the statement.
\end{proof}

Now, let's begin our  proof of Theorem \ref{type_le3}.
First we treat the case with even $\Fr(H)$.

\begin{Proposition}\label{even_type_le2}
We assume $e=4$. If $\Fr(H)$ is even and if $\Fr(H)/2 \in \PF(H)$, then $\type(H)=2$.
That is, if $e=4$ and $H$ is almost symmetric of even type, then $\type(H)=2$.
\end{Proposition}
\begin{proof}  It suffices to show that $\RF(\Fr(H)/2)$ is as in \ref{normal} in Theorem \ref{type2}.
Then we have seen that $\type(H)=2$ by Theorem \ref{Type_even}.
We recall the proof of Theorem \ref{type2} and show that $b=\al_{42}>0$ and $a =\al_1-1$
in the following matrix $(\dagger)$.
\[
(\dagger) \quad \RF(\Fr(H)/2) =  \left( \begin{array}{cccc} -1 & \alpha_2-1 & 0 & 0\\
0 & -1 & \al_3-1 & 0 \\
a & 0 & -1 & \al_4-1\\
\al_1-1 &
 b=\al_{42} & 0 & -1 \end{array}\right),
\]

If we assume $b=0$, adding the 2nd and 4th rows of $(\dagger)$, we get
\[\Fr(H) + n_2+n_4 = (\al_1-1)n_1+(\al_3-1)n_3.\]
Since $(\al_1-1)n_1- n_2, (\al_3-1)n_3-n_4\not\in H$, by Lemma \ref{Apery},
 \begin{eqnarray*}
 (\al_1-1)n_1&=& f + n_2=\Fr(H)/2+n_4  \quad \text{ and}\\
 (\al_3-1)n_3&=& f' + n_4= \Fr(H)/2+n_2
 \end{eqnarray*}
for some $f,f'\in \PF'(H)$ with $f+f'=\Fr(H)$.
Also, we have $f + n_1= \al_1n_1 - n_2 = (\al_2-2)n_2 +n_4$ and likewise,
$f' + n_3= an_1 + n_2 + (\al_4-2)n_4$. Now by Lemma \ref{special row},
there should be  special rows $(\al_2 -1) {\bld e}_2 - {\bld e}_j, (\al_4 -1) {\bld e}_4 - {\bld e}_k$
on either $\RF(f)$ or $\RF(f')$.  We write $\RF(f) = (m_{ij})$ and $\RF(f')=(m'_{ij})$.  We have
seen $m_{13}=m_{23}=0$. Hence we must have $m_{43}>0$ by Lemma \ref{a_ij>0}.
Then we must have $m'_{34} = \al_4-2=0$, giving $\al_4=2$.
Also, since $m_{21}=\al_1-1>0$, $m'_{12}=0$.  Hence only possibility of  $(\al_2 -1) {\bld e}_2 - {\bld e}_j$
is the 3rd row of $\RF(f)$, giving $m_{31}=m_{34}=0, m_{32}=\al_2-1$ and
$(\al_2-1)n_2= f + n_3$. Since $(\al_2-1)n_2= \Fr(H)/2 + n_1$ and $(\al_1-1)n_1= f + n_2=\Fr(H)/2+n_4$,
we have $n_1+n_2= n_3+n_4$.
Then since $\Fr(H)/2 - f' = n_4-n_2= n_1-n_3$, we have $n_4= (\al_4-1)n_4 = f' + n_1$.
We have seen $f + n_1= (\al_2-2)n_2 +n_4$ above. Substituting $n_4= (\al_4-1)n_4 = f' + n_1$
we get $f + n_1= (\al_2-2)n_2 +n_4 = (\al_2-2)n_2 + f' + n_1$ and then $f = (\al_2-2)n_2 + f'$,
contradicting $f\ne f'$ and $f' \in \PF(H)$.  Tuus we have showed $b=\al_{42}>0$.
\par
We have seen at the end of the proof of Theorem \ref{type2} (3),
$\al_{12}=\al_2-1-b >0$, $\al_{13}=0$, $\al_{14}=1$,  $\al_{21}=1$, $\al_{23}=\al_3-1$, $\al_{24}=0,$
$\al_{31}=a$, $\al_{32}=1$, $\al_{34}=\al_4-1$, $\al_{41}=\al_1-1-a$, $\al_{42}=b$, $\al_{43}=1.$

Since $b>0$, we can assert that $\Fr(H)+n_2$ has UF by Lemma \ref{formula_Ap}.

\medskip
Next we will show $a=\al_1-1$.  We have seen that $a\le \al_1-1$ in the proof of  Theorem \ref{type2}.
If $a< \al_1-1$ then we have $\al_{41}>0$ and hence by by Lemma \ref{formula_Ap},
$\Fr(H)+n_1$ has UF.  We will show that this leads to a contradiction.\par

Adding 1st and 2nd  rows of $(\dagger)$, we have
\[\Fr(H) + n_1 =  (\al_2-2)n_2+ (\al_3-1)n_3.\]

Since this expression is unique, $\Fr(H) + n_1 - n_4\not\in H$ and thus
$\Fr(H) + n_1 - n_4 \in \PF'(H)$ by Lemma \ref{Apery}.  We put
\[f = \Fr(H) + n_1 - n_4, \quad f' = \Fr(H) - f = n_4 -n_1.\]
Since $H$ is almost symmetric, $f' \in \PF'(H)$.
Then by Proposition \ref{f + n_k = bn_i}, $\al_4=2$.
Then
\[f' + n_4 = 2n_4 - n_1 = \al_4 n_4 - n_1= (\al_1-2-a)n_1+bn_2+n_3.\]
Then by Corollary \ref{0_in_row}, since $b>0$, we must have $\al_1-2-a =0$; $a=\al_1-2$. \par

Since $(4,2), (4,3)$ entries of $\RF(f')$ are both positive, $(2,4), (3,4)$ entries of $\RF(f)$ must be $0$
by Lemma \ref{symm}.  Then by Lemma \ref{a_ij>0}, $(1,4)$ entry of $\RF(f)$ is positive,
which induces that $f + n_1 \ge_H n_4 = f' + n_1$. Then we have $f \ge_H f'$, contradicting our
assumption  $f,f' \in \PF'(H)$.  Thus we have shown that if $\Fr(H)/2 \in \PF(H)$, then
$\RF(\Fr(H)/2)$ is as in Theorem \ref{type2}.  We have shown that if $H$ is almost symmetric
with even $\Fr(H)$, then $\type(H)=2$.
\end{proof}

\begin{proof}[Proof of Theorem \ref{type_le3}]
If $H$ is almost symmetric and $\Fr(H)$ is even, then we have shown in Proposition
\ref{even_type_le2}
that $\type(H)=2$.   So, in the rest of this section,  we will assume that $\Fr(H)$ is  odd.

Let $H=\langle n_1,n_2,n_3,n_4\rangle$ be almost symmetric and let $\{i,j,k,l\}$ be a permutation of
$\{1,2,3,4\}$. \par

Our tool is Moscariello's $\RF$-matrices and especially the
special rows of those matrices.

By Lemma  \ref{special row} there are
 at least $4$ special rows in $\RF(f)$ and $\RF(f')$ together,
 if  $f+f'= \Fr(H)$.  On the other hand, we showed in Proposition
 \ref{f+n_k= f' + n_l} that for fixed $n_i$, there exists at most
 $2$ relations of type $f+ n_k=(\al_i-1)n_i$.  Since the possibility
 of $n_i$ is $4$, there exists at most $8$ special rows in
 $\RF(f)$ for all possibilities of $f\in \PF'(H)$. This implies
 the cardinality of $\PF'(H)$ is at most $4$ and we have
 $\type(H)\le 5$.  Also, this argument shows that for every $n_i$,
 there are exactly $2$ special rows of type
 $(\al_i -1){\bld e}_i - {\bld e}_k$ in some $\RF(f), f\in \PF'(H)$.

 Now, we assume $\type(H) = 5$ and
 $\PF(H) = \{f_1,f_2,f_2', f_1', \Fr(H)\}$ with
\[f_1<f_2<f_2'< f_1' \quad {\mbox{\rm and}} \quad
f_1+f_1' = f_2+f_2'=\Fr(H).\]

We will show a contradiction.

 We have seen in Lemma \ref{special row}
that for every pair of $f, f'\in \PF(H)$ with $f+f'=\Fr(H)$,
and for every $n_i$, there is a special row of type
$(\al_i-1){\bld e}_i - {\bld e}_k$ for some $k$ on either
$\RF(f)$ or $\RF(f')$.  Hence if we had a relation of type
 \[f+n_k = (\al_i-1)n_i = f' + n_l,\]
 then there will be  $2$ special rows
 $(\al_i-1){\bld e}_i - {\bld e}_k$ in $\RF(f)$ and
 $(\al_i-1){\bld e}_i - {\bld e}_l$ in $\RF(f')$, which
 leads to $5$ special rows on $\RF(f)$ and $\RF(f')$ together,
 which contradicts the fact there are at most $8$ special
 rows in total.

 Hence if there is a relation of type
 \[ f + n_k = (\al_i-1) n_i = f' + n_l,\]
 then $f+f' \ne\Fr(H)$. Since there exists  exactly
 $4$ such pairs of $\{f,f'\}$, namely,
 $\{f_1,f_2\}, \{f_1', f_2\}. \{f_1,f'_2\}, \{f_1', f_2'\}$,
 we must have the following relations for some
 $ n_p,\ldots , n_y\in \{n_1,n_2,n_3,n_4\}$.

Let $\{i,j,k,l\}$ be a permutation of $\{1,2,3,4\}$
and assume that $n_1<n_2<n_3<n_4$.

Now we must have the following relations

\begin{eqnarray}
f_1+ n_p = (\al_i-1)n_i = f_2+ n_q\\
f_1'+ n_r = (\al_j-1)n_j = f_2+ n_s\\
f_1+ n_t = (\al_k-1)n_k = f_2'+ n_u\\
f_1'+ n_x = (\al_l-1)n_l = f_2'+ n_y
\end{eqnarray}
\vskip 0.5cm

From these equations and since $f_2-f_1= f_1' - f_2'$,
we have

\begin{eqnarray}
n_p-n_q= f_2-f_1= f_1'- f_2' = n_y-n_x\\
n_s-n_r= f_1'-f_2= f_2'-f_1= n_t-n_u.
\end{eqnarray}

We divide the cases according to how many among
$\{n_p,n_q, n_x, n_y\}$ and $\{n_t.n_u, n_x,n_y\}$ are
different.
\vskip 0.2cm
 Case 1: First, we assume that $n_p = n_y$ and $n_q= n_x$.
Since we have assumed $\{i,j,k,l\}$ is a permutation of $\{1,2,3,4\}$
and since $n_i, n_l$ must be different from $n_p, n_q$, we must have
$\{n_l, n_k\} = \{n_p, n_q\}$.  Then from equations (4), (5), $n_r=n_u
< n_s= n_t$ and $\{n_i, n_l\} =\{n_r, n_s\}$.

Now for the moment, assume $n_i > n_l$ (hence $n_i= n_s= n_t$ and
$n_l = n_r = n_u$)  and we will deduce a contradiction. The assumption
 $n_i < n_l$ leads to a contradiction similarly.
\par
We check $\RF$ matrices of
$f_1,f_2, f_2', f_1'$ with respect to $\{n_i,n_p, n_q, n_l\}$.
Equations (3), (6) show us the $p$-th row of $\RF(f_1)$ is
$(\al_i-1, -1,0,0)$ and since  $\RF(f'_1)_{qi}=\al_i-1$,
$\RF(f_1)_{iq}=0$ by  Lemma \ref{symm}. Hence by Lemma
\ref{a_ij>0}, $\RF(f_1)_{lq}>0$. From equation (5), to have
$\RF(f_1)_{lq}>0$, we must have $n_k = n_q$ and then
$n_j = n_p$. Then from  (3), we have
\[\al_i n_i = (f_1+ n_i) + n_p = (f_2 + n_i) + n_q.\]
But from  (4), (5), we know $f_1+ n_i = (\al_q-1)n_q,
f_2+n_i = (\al_p-1)n_p$ and then we have
$(\al_q-1)n_q+ n_p = (\al_p-1)n_p + n_q$, getting $\al_p=\al_q=2$ and
from (4), (5),
$n_p = f_1' + n_l = f_2 + n_i, n_q = f_1 + n_i = f_2' + n_l$.
Then from (3), (6), we get $\al_in_i = (f_1+n_i) + n_p = n_p+n_q
= (f_1'+ n_p) +n_l= \al_l n_l$.

Now we compute $f_1 + n_l$. We know that $f_1+n_l < n_p, n_q$.
Then we must have $f_1 + n_l = c n_i$ for some positive integer $c$.
But by Lemma \ref{f + n_k = bn_i}, $c\ge \al_i -1$. Then we have
$f_1 + n_l + n_i \ge \al_in_i = n_p+n_q$, contradicting $n_p = f_1' +n_l$ and
$n_q = f_1+n_i$. Thus Case 1 does not occur.
\bigskip

Case 2:  If $\sharp \{n_p,n_q,n_x, n_y\}=3$,
then  either $n_p=n_x$ or $n_q=n_y$ and hence
either $2 n_p = n_q + n_y$ or $2 n_q= n_p+ n_x$,
having $\al_p=2$ or $\al_q=2$. For the moment we assume
$2 n_q= n_p+ n_x$ and $\al_q=2$.   Then from (3) - (6),
for some $f,f'\in \PF'(H)$,
\[(*) \quad (\al_q-1)n_q= n_q =f+n_v = f' + n_w.\]
Since $n_q$ is not the biggest among $n_1,\ldots , n_4$
and bigger than the other $2$, $n_q=n_3$.
We assume $n_1<n_2 < n_3=n_q < n_4$ and again compute
$n_1 + f_1.$ By (*), $n_1+ f_1< n_3$ and we muct have
$n_1+ f_1 = c n_2$ for some positive integer $c$.
We have seen $c \al_2-1$ and if $n_1+ f_1 = (\al_2-1)n_2$, then by
(3) - (6), $n_1+ f_1= (\al_2-1)n_2 = n_w+ f$ but that is impossible since
$n_1+ f_1< n_3$.  If $c\ge \al_2$ we get also a contradiction since
$n_1+f_1$ cannot contain $n_1,n_3,n_4$.
Thus Case 2 does not occur. It is easy to see that $\sharp\{n_r,n_s,n_t, n_u\} =3$
leads to a contradiction, either.
 \vskip 0.2cm

Case 3:   To prove the Theorem  \ref{type_le3}, it suffices to get a contradiction
assuming   $\{n_p,n_q,n_x, n_y\}$ and $\{n_r,n_s,n_t, n_u\}$ are different elements.
By (3) - (6), we may assume $n_1<n_2<n_3<n_4$ with $n_2-n_1= n_4-n_3 = f_2-f_1$
and $n_3-n_1=n_4- n_2 = f_2' - f_1$.  Hence $n_4-n_1 = (f_1'-f_2) + (f_2-f_1) = f_1'-f_1$
and we have $f_1 +n_1= f_1' + n_4$.  Also, since $f_1 +n_4$ is not of the type  $(\al_w-1)n_w$,
we have
\[f_1+n_4= f_1' + n_1 = b n_2 + c n_3\]
with $c,d >0$.  Then from Lemma \ref{symm} we must have $\RF(f_1)_{2,1} = \RF(f_1)_{3,1}=0$.
Since the 4th row of $\RF(f_1) = (0,b,c,0)$, every component of the 1st column of $\RF(f_1)$ is $0$,
contradicting Lemma  \ref{a_ij>0}.
This finishes our proof of Theorem \ref{type_le3}.
\end{proof}

\section{On the free resolution of $k[H]$.}
\label{sec:6}
Let as before $H=\langle n_1,\ldots,n_e\rangle$ be a numerical semigroup and  $K[H]=S/I_H$ its  semigroup ring over $K$.

We are interested in the minimal graded free $S$-resolution $(\FF, d)$ of $K[H]$. For each $i$,  we have $F_i = \bigoplus_j S(- \beta_{ij})$,  where the $\beta_{ij}$  are the graded Betti numbers of $K[H]$. Moreover,  $\beta_i = \sum_j \beta_{ij}= \rank(F_i)$ is the $i$th Betti number of $K[H]$.
Note that $\projdim_SK[H]=e-1$ and that $F_{e-1} \cong \bigoplus_{f\in \PF(H)} S( - f - N)$, where we put $N=\sum_{i=1}^e n_i$

Recall from Section~1 that $H$ is almost symmetric, or, equivalently, $R$ is almost Gorenstein 
  if the cokernel of a natural morphism
   $$R \to \omega_R( -\Fr(H))$$
  is annihilated by the graded maximal ideal of $K[H]$. In other words,
  there is an exact sequence of graded $S$-modules
$$0 \to R \to \omega_R( -\Fr(H))   \to \bigoplus_{f \in \PF(H), f\ne \Fr(H)} K(-f)\to 0. $$
Note that, we used the symmetry of $\PF(H)$ given in Lemma~\ref{Nari} when $H$ is almost symmetric.

Since $\omega_S\cong S(-N)$, the minimal free resolution of $\omega_R$ is given by
the  the $S$-dual $\FF^{\vee}$  of $\FF$ with respect to $S(-N)$.
Now, the injection $R \to \omega_R (-\Fr(H))$
 lifts to a morphism $\varphi : \FF \to \FF^{\vee}(- \Fr(H))$,  and the resolution of the  cokernel of $R\to
K_R(-\Fr(H))$ is given by the mapping cone $\MC(\varphi)$ of $\varphi$.

On the other hand,  the free resolution of the residue field $K$ is given by the Koszul complex
$\KK= \KK(x_1,\ldots , x_e;K)$. Hence we get

\begin{Lemma}\label{Koszul}
The mapping cone $\MC(\varphi)$ gives a (non-minimal) free $S$-resolution of
$\bigoplus_{f\in \PF(H), f\ne \Fr(H)} K(-f)$.  Hence, the minimal free resolution obtained from
$\MC(\varphi)$ is isomorphic to $\bigoplus_{f\in \PF(H), f\ne \Fr(H)} K(-f).$
\end{Lemma}

Let us discuss the case $e=4$ in more details. For $K[H]$ with $t=\type(K[H])$  we have the graded minimal free resolution
\[
0\to \Dirsum_{f\in \PF(H)}S(-f-N)\to \Dirsum_{i=1}^{m+t-1}S(-b_i)\to \Dirsum_{i=1}^mS(-a_i)\to S\to K[H]\to0
\]
of $K[H]$. The dual with  respect to $\omega_S=S(-N)$ shifted by $-F(H)$ gives the exact sequence
\begin{eqnarray*}
0\to S(-\Fr(H)-N)&\to & \Dirsum_{i=1}^mS(a_i-F(H)-N)\to \Dirsum_{i=1}^{m+t-1}S(b_i-F(H)-N)\\
&\to & \Dirsum_{f\in \PF(H)}S(f-\Fr(H))\to \omega_{K[H]}(-\Fr(H))\to 0.
\end{eqnarray*}

Considering the fact that for the map $\varphi\: \FF\to \FF^\vee$ the component  $\varphi_0\: S\to   \Dirsum_{f\in \PF(H)}S(f-F(H))$ maps $S$ isomorphically to $S(\Fr(H)-\Fr(H))=S$, these two terms can be canceled against each others in the mapping cone. Similarly,via $\varphi_4: \Dirsum_{f\in \PF(H)}S(-f-N)\to S(-\Fr(H)-N)$ the summands  $S(-\Fr(H)-N)$ can be canceled. Observing then that $\PF'(H)=\{\Fr(H)-f\: f\in \PF'(H)\}$, we  obtain the reduced mapping cone
\begin{eqnarray*}
0&\to& \Dirsum_{f\in \PF'(H)}S(-f-N)\to \Dirsum_{i=1}^{m+t-1}S(-b_i)\to \Dirsum_{i=1}^mS(-a_i)\dirsum \Dirsum_{i=1}^mS(a_i-F(H)-N)\\
&\to&  \Dirsum_{i=1}^{m+t-1}S(b_i-F(H)-N)\to \Dirsum_{f\in \PF'(H)}S(-f)\to \Dirsum_{f\in \PF'(H)}K(-f)\to 0.
\end{eqnarray*}
which provides a graded free resolution of $\Dirsum_{f\in \PF'(H)}K(-f)$.
Comparing this resolution with  the minimal  graded free resolution of $\Dirsum_{f\in \PF'(H)}K(-f)$, which is
\begin{eqnarray*}
0&\to &\Dirsum_{f\in \PF'(H)}S(-f-N)\to  \Dirsum_{f\in \PF'(H)\atop 1\leq i \leq 4}S(-f-N+n_i)\to \Dirsum_{f\in \PF'(H)\atop 1\leq i<j\leq 4}S(-f-n_i-n_j)\\
&\to &  \Dirsum_{f\in \PF'(H)\atop 1\leq i\leq 4}S(-f-n_i)\to \ \Dirsum_{f\in \PF'(H)}S(-f)\to \Dirsum_{f\in \PF'(H)}K(-f)\to 0,
\end{eqnarray*}
we notice that  $m\geq 3(t-1)$. If $m=3(t-1)$,  then   reduced mapping cone provides a graded minimal free resolution of $\Dirsum_{f\in \PF'(H)}K(-f)$. Also, if $m = 3(t-1)+s$ with $s>0$, then there should occur
$s$ cancellations in the mapping $\varphi : \Dirsum_{i=1}^mS(-a_i)
\to  \Dirsum_{i=1}^{m+t-1}S(b_i-F(H)-N)$.

\medskip
A comparison of the mapping cone with the graded minimal free resolution of $\Dirsum_{f\in \PF'(H)}K(-f)$ yields the following numerical result.

\begin{Proposition}
\label{comparison}
Let $H$ be a $4$-generated  almost symmetric numerical semigroup of type $t$.  Then
putting $m_0=  3(t-1)$, we have $m=\mu_R(I_H) = m_0+s$ with $s\ge 0$.
Moreover,  with the notation introduced,  we can put $\{a_1, \ldots, a_{m_0}, \ldots , a_m=a_{m_0+s}\}$
and  $\{b_1,\ldots,b_{m_0+t-1},\ldots , b_{m+t-1}\}$ so that
one has  the following equalities of multisets:
\begin{eqnarray*}
&&\{a_1,\ldots, a_m\}\union \{\Fr(H)+N-a_1,\cdots, \Fr(H)+N-a_m\}\\
&&=\{f+n_i+n_j \:\; f\in \PF'(H), 1\leq i<j\leq 4\},
\end{eqnarray*}
and
\[
\{b_1,\ldots,b_{m_0+t-1}\}=\{f+N-n_i\:\; f\in  \PF'(H), 1\leq i\leq 4\}.
\]
and if $s>0$,
\[
a_{m_0+j} = \Fr(H) - b_{m_0+ t-1 +j} , (1\leq j \leq s).
\]
\end{Proposition}

\begin{ex}\label{a&b} (1) Let $H = \langle  5,6,7,9 \rangle$.
Then $H$ is pseudo-symmetric with $\PF(H) = \{ f=4, \Fr(H)= 8\}$ and  we see
\[\{a_1,\ldots, a_5\} = \{ 15,16,18; 12, 14\}, \{b_1,\ldots , b_6\} = \{22, 24, 25, 26; 23,21\},\]
where we see $15 = f + n_1+n_2, 16= f;n_1+n_3, 16= f + n_1+n_4; \Fr(H)+N = 35$ and
$35- a_4= b_5, 35 - a_5 = b_6$.\par
(2) Let $H = \langle  18,21,23,26 \rangle$. Then $\PF(H) = \{ 31, 66, 97\}$ showing that
$H$ is AS of type $3$. Then we see $\mu(I_H) = 7 = 3(t-1) +1$ and
 \[\{a_1,\ldots, a_7\} = \{ 72, 75, 78, 105, 110, 115; 44\}, \]
\[ \{b_1,\ldots , b_9\} = \{93, 96, 98, 101, 128, 131, 133, 136; 141 \},\]
where $141 = \Fr(H) + N - a_7$.
\end{ex}
\medskip
Now let us assume that $t=3$,  and that $m=6$.
\medskip
An example of an almost symmetric  $4$-generated numerical semigroup  of type $3$ with $6$  generators for $I_H$ is the semigroup $H=\langle  5, 6, 8, 9 \rangle$. In this example $I_H$ is generated by
\begin{eqnarray*}
\phi_1&=&x_1^3-x_2x_4,\; \phi_2=x_2^3-x_1^2x_3,\;  \phi_3=x_3^2-x_1^2x_2,\;  \phi_4=x_4^2-x_2^3,\\
\phi_5&=&x_1x_2^2-x_3x_4,\; \phi_6= x_1x_4-x_2 x_3.
\end{eqnarray*}
\medskip
We have $\PF(H)=\{3,4,7\}$, and the $\RF$-matrices of $H$ for $3$ and $4$ are
\[
\RF(3) =  \left( \begin{array}{cccc} -1 & 0 & 1 & 0\\
0 & -1 & 0 & 1 \\
1 & 1 & -1 & 0\\
0 & 2 & 0 & -1 \end{array}\right),
\quad
and
\quad
\RF(4) =  \left( \begin{array}{cccc} -1 & 0 & 0 & 1\\
2 & -1 & 0 & 0 \\
0 & 2 & -1 & 0\\
1 & 0 & 1 & -1 \end{array}\right).
\]
\medskip
The $\RF$-relations resulting from $\RF(3)$ (which are obtained by taking for each $i<j$ the difference  of the $i$th row and the $j$th row of the matrix ) are
\begin{eqnarray*}
\phi_{12}&=&x_2x_3-x_1x_4,\; \phi_{13}=x_3^2-x_1^2x_2,\; \phi_{14}=x_3x_4-x_1x_2^2,\; \phi_{23}= x_3x_4-x_1x_2^2,\\
\phi_{24}&=& x_4^2-x_2^3,\; \phi_{34}=x_1x_4-x_2x_3,
\end{eqnarray*}
while the $\RF$-relations resulting from $\RF(4)$ are
\begin{eqnarray*}
\psi_{12}&=&x_2x_4-x_1^3,\; \psi_{13}=x_3x_4-x_1x_2^2,\; \psi_{14}=x_4^2-x_1^2x_3,\; \psi_{23}=x_2^3 -x_1^2x_3,\\
\psi_{24}&=& x_1x_4-x_2x_3,\; \psi_{34}=x_2^2x_4-x_3^2x_1.
\end{eqnarray*}
We see that $\deg \phi_{ij}=3+n_i+n_j$ for all $i<j$, except for $\phi_{34}$ for which we have  $\deg \phi_{34}=14<20= 3+8+9$. Similarly,
$\deg \psi_{ij}=4+n_i+n_j$  for all $i<j$, except for $\psi_{24}$ for which we have  $\deg \psi_{24}=14< 19=4+6+9$.

Comparing the $\RF$-relations with the generators of $I_H$ we see that
\begin{eqnarray*}
&&\phi_1=-\psi_{12},\;  \phi_2= \psi_{23},\; \phi_3=  \phi_{13},\\
&&\phi_4= \phi_{24},\; \phi_5= -\phi_{14}=-\psi_{13} ,\; \phi_6 = -\phi_{12}=\phi_{34}=\psi_{24}.
\end{eqnarray*}
In this example we see  that the $\RF$-relations generate $I_H$.

\medskip
The next result shows that this is always the case for such kind of numerical semigroups

\begin{Theorem}
\label{rfrelations}
Let $H$ be a $4$-generated  almost symmetric numerical semigroup of type $t$ for which $I_H$ is generated by $m=3(t-1)$ elements. Then $I_H$ is generated by $\RF$-relations.
\end{Theorem}

\begin{proof}
By  Lemma~\ref{conditionforquestion} it suffices to show that $I_H$ admits a system of generators $\phi_1,\ldots, \phi_r$  such that for each $k$ there exist $i<j$ and $f\in \PF'(H)$  such that $\deg \phi_k=f+n_i+n_j$ and $\phi_k=u-v$, where $u$ and $v$ are monomials such that $x_i|u$ and $x_j|v$, or $x_j|u$ and $x_i|v$.

Consider the chain map $\varphi : \FF \to \GG$ with $\GG=\FF^{\vee}(- \Fr(H))$ resulting from the inclusion $R \to \omega_R( -\Fr(H))$:
\[
\begin{CD}
0@>>> F_3@>\partial_3 >> F_2@>\partial_2 >> F_1@>\partial_1 >> F_0\\
 @.      @VV\varphi_3 V @VV\varphi_2 V @VV\varphi_1 V @VV\varphi_0 V\\
0@>>> G_3@>d_3 >> G_2@>d_2 >> G_1@>d_1 >> G_0. \\
\end{CD}
\]
The assumption of the theorem implies that the reduced mapping cone of $\varphi$ is isomorphic as a graded complex to a direct sum of Koszul comlexes with suitable shifts, as described above. Thus we obtain a commutative diagram
\[
\begin{CD}
K_2@>\kappa_2 >> K_1@>\kappa_1 >> K_0@.\\
@VV\alpha_2 V @VV\alpha_1 V @VV\alpha_0 V@. \\
F_1\dirsum G_2@>>(-\varphi_1,d_2) > G_1@>>\bar{d_1} > \overline{G_0}@= G_0/F_0\\
\end{CD}
\]
of graded free $S$-modules, where the top complex is the begin of the direct sum of Koszul complexes, and where each $\alpha_i$ is a graded isomorphism.

We choose suitable bases for the free modules involved in this diagram. The free module
$K_0$ (resp. $K_1$)  admits a basis $\{ e_f \: f\in \PF'(H)\}$ (resp.
$\{e_{f,i}\: f\in \PF'(H),\; i=1,\ldots,4\}$) with $\deg e_f = f$,
$\deg e_{f,i}=f+n_i$ and such that $\kappa_1(e_{f,i})=x_ie_f$.
Then a basis for $K_2$ is given by the wedge products $e_{f,i}\wedge e_{f,j}$ with $\deg(e_{f,i}\wedge e_{f,j})=f+n_i+n_j$.

On the other hand, $F_1$ admits a basis $\epsilon_1,\ldots, \epsilon_m$,  where  each $\epsilon_k$
  has a degree of the form $f+n_i+n_j$ for some $f\in \PF'(H)$ and some $i<j$ since $\alpha_2$ is an
  isomorphism of graded complexes.
  Moreover, $\partial_1(\epsilon_k)=\phi_k$,  where $\phi_k=u_k-v_k$ is a binomial with $\deg u_k=\deg v_k=\deg \phi_k=\deg \epsilon_k$.

Let $\alpha_2(e_{f,i}\wedge e_{f,j})=\epsilon_{f,ij}+\sigma_{f,ij}$ with $\epsilon_{f,ij}\in F_1$ and $\sigma_{f,ij}\in G_2$ for $f\in \PF'(H)$ and $i<j$. Then the elements  $\epsilon_{f,ij}$ generate $F_1$. Moreover, we have
\[
-\varphi_1(\epsilon_{f,ij})+d_2(\sigma_{f,ij})=x_i\alpha_1(e_{f,j}) - x_j\alpha_1(e_{f,i})\subset (x_i,x_j)G_1.
\]
Since $\bar{d_1}(-\varphi_1(\epsilon_{f,ij})+d_2(\sigma_{f,ij}))=0$. it follows that
\[
-\varphi_0\partial_1(\epsilon_{f,ij})=d_1(-\varphi_1(\epsilon_{f,ij})+d_2(\sigma_{f,ij}))
=d_1(-\varphi_1(\epsilon_{f,ij}))\in (x_i,x_j)\varphi_0(F_0).
\]
If follows that \begin{eqnarray}
\label{ij}
\partial_1(\epsilon_{f,ij})\subset (x_i,x_j)F_0=(x_i,x_j),
\end{eqnarray}
for all $f\in \PF'(H)$ and $i<j$.

Since the elements $\epsilon_{f,ij}$ generate $F_1$, it follows that  the elements

$\partial_1(\epsilon_{f,ij})$ generate $I_H$. 

To show that $I_H$ is generated by RF-relations, it suffices to show the  following.

\begin{enumerate}
\item  $I_H$ admits a system of binomial generators $\phi_1,\ldots,\phi_m$
 such that  for each $k$ there exist $i<j$ and $f\in \PF'(H)$  with $\deg \phi_k=f+n_i+n_j$

\item  $\phi_k=u-v$, where $u$ and $v$ are monomials such that $x_i|u$ and $x_j|v$.
\end{enumerate}

For given $f\in\PF'(H)$ and $i<j$, let $\epsilon_{f,ij}=\sum_{k=1}^m\lambda_k\epsilon_k$. Then $\lambda_k=0$ if  $\deg \epsilon_k \neq  f+n_i+n_j$, and

$\partial_1(\epsilon_{f,ij})=\sum_{k=1}^m\lambda_k \phi_k=\sum_{k=1}^m\lambda_k(u_k-v_k)$
with $\lambda_k\neq 0$ only if degree $\deg u_k=\deg v_k=f+n_i+n_j$.
This sum can be rewritten as $\sum_{k=1}^r\mu_kw_k$ with $\sum_{k=1}^r\mu_k=0$,
and pairwise distinct monomials $w_k$ with  $\{w_1,\ldots,w_r\}\subset \{u_1,v_1,\ldots, u_m,v_m\}$
 and $\deg w_k=f+n_i+n_j$ for $i=1,\ldots,r$. Provided $\partial_1(\epsilon_{f,ij})\neq 0$,
 we may assume that $\mu_k\neq 0$ for all $k$. Then $\sum_{k=1}^r\mu_kw_k=\sum_{k=2}^r\mu_k(w_k-w_1)$. Since $\deg w_k=\deg w_1$ for all $k$, it follows that  $w_k-w_1\in I_H$ for all $k$. Moreover, since the $\partial_1(\epsilon_{f,ij})$ generate $I_H$, we see that the binomials $w_k-w_1$
  in the various  $\partial_1(\epsilon_{f,ij})$  altogether generate $I_H$.
We also have that $\sum_{k=2}^r\mu_k(w_k-w_1)\subset (x_i,x_j)$. We may assume
 that $x_j|w_1$, and $x_j$ does not divide  $w_2,\ldots,w_s$ while $x_j|w_k$ for $k=s+1,\ldots,r$.
 Then $x_i|w_k$ for $k= 2,\ldots,s$ and
 $\partial_1(\epsilon_{f,ij})+\mm I_H=\sum_{i=2}^s\mu_k(w_k-w_1)+\mm I_H$.
 It follows that modulo $\mm I_H$, the ideal $I_H$ is generated by binomials $\phi=u-v$
 for which there exists $f\in\PF'(H)$  and
$i<j$ such that $\deg \phi=f+n_i+n_j$ and $x_i|u$ and $x_j|v$. By Nakayama,
the same is true for $I_H$, as desired.
\end{proof}

\subsection{$7$-th generator of $I_H$.}

In this subsection, we will show that if $H$ is almost symmetric
generated by $4$ elements, then $I_H$ is generated by $6$ or
$7$ elements,  and if $I_H$ is generated by $7$ elements,
we can determine such $H$.

We put $H=\left< n_1,n_2,n_3,n_4\right>$
with $n_1<n_2 <n_3 <n_4$ and put $N=\sum_{i=1}^4 n_i$.
Also, we always assume $H$ is almost symmetric and
$\PF(H) = \{ f, f', \Fr(H)\}$ with $f+f'= \Fr(H)$.

\begin{Theorem}\label{7th}
 If $I_H$ is  minimally generated by more
 than $6$ elements, then $I_H$ is generated by $7$ elements
  and we have the relation $n_1+n_4= n_2+n_3$.
\end{Theorem}
\begin{proof}
By Lemma \ref{Koszul}, if we need more than $6$ generators
for $I_H$, then there must be a cancellation in the mapping
\[\phi_2 : F_1 \to F_2^\vee.\]
Namely, there is a monomial generator $g$ of $I_H$ and a free base $e$ of $F_2$ with
\[\deg e =  \Fr(H) + N - \deg(g).\]

On the other hand, being a base of $F_2$, $e$ corresponds to
a relation
\[ (**) \quad \sum_i g_i y_i = 0,\]
where $g_i$ are generators of $I_H$ and $\deg y_i = h_i\in H_+$.
 It follows that for every $i$ with $y_i\ne 0$, we have
 \[\deg e = \deg g_i + h_i.\]

Note that by Lemma \ref{Koszul}, we have $6$ minimal
 generators of $I_H$, whose degree is of the form
 \[ f + n_i + n_j \quad (f\in \PF'(H)).\]

 Now, by Lemma \ref{special row}, for every $i$,
 there is $f\in \PF'(H)$ and some $n_k$ such that
 $(\al_i -1)n_i = f + n_k$, or, $\al_in_i = f + n_i + n_k$.
 Since our $g$ has degree not of the form $ f + n_i + n_k$,
 we may assume that
 \[g = x_i^ax_j^b - x_k^cx_{\ell}^d\]
 for some permutation $\{i,j,k,\ell\}$ of $\{1,2,3,4\}$
  with positive $a,b,c,d$. Our aim is to show
  $a=b=c=d=1$.
\end{proof}

 We need a lemma.

 \begin{lem}\label{relation} In the relation $(**)$ above,
 the following  holds:
 \begin{enumerate}
\item[{\em (i)}]  If $y_i\ne 0$ in $(**)$, then at most $3$ $x_i's$
appear in $g_i$.
 \item[{\em (ii)}]  At least $3$ non-zero  terms appear in $(**)$.
 \end{enumerate}
 \end{lem}
\begin{proof} (i) If $g$ contains $4$ $x_i$'s, then either
 $g = x_r^{\al_r} - x_s^{\beta}x_t^{\gamma}x_u^{\delta}$
 or $g= x_r^{a'} x_s^{b'}- x_t^{c'}x_u^{d'}$ for some permutation of
 $x_i$'s. In the first case, if $n_r \in \{n_i, n_j\}$, then
 we have
 \[ \Fr(H) + N = an_i + bn_j + {\beta}n_s + {\gamma}n_t +
 {\delta}n_u,\]
 which will give $\Fr(H)\in H$, a contradiction !
 In the second case, since we have
 \[\Fr(H) + N -\deg(g) = \deg(g_i) + h_i\]
and it is easy to see that $\deg(g) + \deg(g_i) + h_i\ge_H N$,
which again deduce $\Fr(H)\in H$.\par

(ii) Since the minimal generators of $I_H$ are
irreducible in $S$, if there appear only $2$ $g_i$'s in $(**)$
like $y_ig_i - y_jg_j = 0$, then $y_i = g_j$ and $y_j = g_i$
up to constant.  Also, by (i), $g_i$ are of the form
$g_i = x_p^{\al_p} - x_r^px_s^q$.  Then
 for $f_1,f_2\in \PF'(H)$
and $n_p,n_q,n_t,n_u$, we have
$\deg g_i = f_1+ n_p+n_q, \deg g_j = f_2+ n_r+n_s$.
Since $\deg g_i = \al_p n_p$ and $\deg g_j = \al_rn_r$
for $n_p\ne n_r$, there are at least $3$ different elements
among $\{ n_p, n_q,n_r, n_s\}$.  Then the relation

\[ \Fr(H) + N -\deg(g) = (f_1+ n_p+n_q) + (f_2+ n_t+n_u)\]
will deduce $\Fr(H)\in H$.
\end{proof}

\begin{proof}[Proof or Theorem \ref{7th}]
Now, assume that $\Fr(H) + N - \deg(g) = f + n_p + n_q + h$.
Then we deduce
\[f' + n_r + n_s = an_i+bn_j + h = cn_k + dn_{\ell} + h.\]

Since $\{n_i, n_j\}$ and $\{n_k, n_{\ell}\}$ are symmetric at
this stage, we may assume $n_r= n_i, n_s=n_k$ and
\[f' + n_k = (a-1) n_i +bn_j +h, \quad f' + n_i = (c-1)n_k +
dn_{\ell}+h.\]

Now, it is clear that $h$ should not contain $n_i$ or $n_k$.
For the moment, assume that $h = m n_j$.
Then we have
\[f' + n_i = m n_j + (c-1)n_k + d n_{\ell}.\]
By Corollary \ref{0_in_row}, $i$-th row of $\RF(f')$ should
contain $0$ and hence we should have $c=1$.
Likewise, if $h = m n_j + m' n_{\ell}$ with $m,m'>0$, then we will
have $a=c=1$.

Now, since at least $3$ non-zero  terms appear in $(**)$,  by
Lemma \ref{relation} we have at least $3$ relations of the type
\[\Fr(H) - \deg(g) = f_i + (n_{p.i} + n_{q,i})\]
for $i= 1,2,3$.  We can assume $f_1=f_2=f$,
 \[f' + n_i + n_k = an_i+bn_j + h = cn_k + dn_{\ell} + h\]
and also $h = m n_j$ with $m >0$. Note that by our discussion
above, we have $c=1$ and
\[f' + n_k = (a-1) n_i +(b+m) n_j ,
\quad f' + n_i =  m n_j + dn_{\ell}.\]

Now, we have the 2nd relation
\[\Fr(H) - \deg(g) = f+ (n_{p.2} + n_{q,2}) + h_2\]
and hence
\[f' + n_t + n_u = an_i+bn_j + h_2 = n_k + dn_{\ell} + h_2.\]
 We discuss the possibility of $n_t, n_u$ and $h_2$.
 Since $h_2\ne h=mn_j$, we must have $\{n_t, n_u\}
 = \{ n_j, n_{\ell}\}$ and we have either

\medskip
\noindent
Case A:  $h_2= m' n_i, d=1$,  or

\medskip
\noindent
Case B: $h_2= m'' n_k, b=1$.

\medskip
 Now, by the argument above, in  Case A,
matrix $\RF(f')$ with respect to $\{n_i,n_j, n_k, n_{\ell}\}$   is
\[ \RF(f') =   \left( \begin{array}{cccc} -1 & m & 0 & d\\
m' & -1 & 1 & 0 \\
a-1 & b+m & -1 & 0\\
a+m' & b-1 & 0 & -1 \end{array}\right), \]

By Lemma \ref{symm}, $(i,j), (i, \ell)$ entries of $\RF(f)$ is $0$
and by Lemma \ref{a_ij>0}, $(i,k)$ component should be
$\al_k-1>0$. This implies that $(k,i)$ entry $a-1$ of $\RF(f')=0$,
thus we obtain $a=1$. Likewise, since $(j,i), (j,k)$ entry of
$\RF(f)$ are $0$, hence $(j, \ell)$ entry is $\al_{\ell} -1>0$,
forcing $(\ell, j)$ entry $b-1$ of $\RF(f')$ to be $0$.
Thus we have $a=b=c=d=1$. We have the same conclusion
in Case B, too.
\end{proof}

\begin{rem}\label{RF-rel-7gen}  {\em Since we get $a=b=1$ from above proof, the relation
$x_ix_j- x_kx_l$ is obtained by taking the difference of 2nd and 4th row
of  $\RF(f')$ above. So, it is  an \lq\lq RF-relation".
Thus combining this with Theorem \ref{rfrelations} and Theorem \ref{7th} we
see that if  $H$ is almost symmetric of type $3$, then $I_H$ is generated by
RF-relations.}
\end{rem}

\section{When is $H+m$ almost symmetric for infinitely many $m$?}
\label{sec:7}

In this section we consider shifted families of numerical semigroups.

\begin{defn}\label{H+m_def}
{\em For $H =\langle n_1,\ldots ,n_e\rangle$, we put $H+m =\langle n_1+m,\ldots ,n_e+m\rangle$.
When we write $H+m$, we assume that $H+m$ is a numerical semigroup, that is,
$\GCD(n_1+m,\ldots , n_e+m)=1$.
In this section, we always assume that $n_1<n_2 < \ldots < n_e$.
We put}
\[s = n_e -n_1,  d = \GCD(n_2-n_1, \ldots , n_e - n_1) \quad {\rm\text and}\quad  s' = s/d.\]
\end{defn}

\medskip
First, we will give a lower bound of Frobenius number of $H+m$.

\begin{prop}
For $m\gg 1$,  $\Fr(H+m) \ge m^2/s $.
\end{prop}

\begin{proof} Note that $\Fr(H+m) \ge \Fr(H'_m)$, where we put
$H'_m= \langle m+n_1,m+n_1+1,\ldots , m+n_e= m+n_1+s\rangle$,
 and it is easy to see that $\Fr(H'_m)\ge m^2/s$.
\end{proof}

The following fact is trivial but very important in our argument.

\begin{lem}\label{al_i(m)}  If $\phi = \prod_{i=1}^e x_i^{a_i} -
\prod_{i=1}^e x_i^{b_i}\in I_H$ is
homogeneous, namely, if $\sum_{i=1}^e a_i =\sum_{i=1}^e b_i$, then
$\phi \in I_{H+m}$ for every $m$.
\end{lem}

We define $\alpha_i(m)$ to be the minimal positive integer such that
$$\al_i(m)  (n_i +m) = \sum_{j=1, j\ne i}^e \al_{ij}(m) (n_j+m).$$

\begin{lem}\label{homogeneous} Let $H+m$ be as in Definition \ref{H+m_def}. Then,
if $m$ is sufficiently big with respect to $n_1,\ldots , n_e$, then
$\al_2(m),\ldots , \al_{e-1}(m)$ is constant,
$\al_1(m)\ge   (m+n_1)/s'$ and $\al_4(m)\ge (m+n_1)/s'-1$.
Moreover, there is a constant $C$ depending only on $H$ such that
$\al_1(m)-(m+n_1)/s' \le C$ and $\al_4(m)-(m+n_1)/s' \le C$.
\end{lem}

\begin{proof}  It is obvious that there are some {\it homogenous} relation $\phi\in I_H$
of type $\phi=  x_i^{a_i} - \prod_{j=1, j\ne i}^e x_j^{b_j}$ if $i \ne 1,e$.
Thus for $m\gg 1$, $\al_i(m)$ is the minimal $a_i$ such that there exists a homogeneous equation
of type  $\phi=  x_i^{a_i} - \prod_{j=1, j\ne i}^e x_j^{b_j} \in I_H$.\par

If $\al_1(m) (m+n_1)= \sum_{i=2}^e a_i(m+n_i)$, then obviously, $\al_1(m) \ge \sum_{i=2}^e a_i+1$.
Moreover, if $d>1$, then since we should have
$\GCD(m+n_1,d)=1$
 to make a numerical semigroup,
we have $\al_1(m) \equiv \sum_{i=2}^e a_i$ (mod $d$).
 Hence $\al_1(m)\ge \sum_{i=2}^e a_i+d$.

We can compute $\al_1(m)$ in the following manner.
We assume that $m$ is sufficiently large and define $m'$ by the equation
 \[(m+n_1)=s' m' - r  \quad {\text with} \quad 0\le r<s',\]
 then  $m' (m+ n_e) - (m' +d) (m+n_1) = dr \ge 0 $ and
 also, for an integer $c>0$, we have $ (m'+c) (m+ n_e) - (m' +d+c) (m+n_1)
 = dr +cs.$  Take $c$ minimal so that
$(m'+c) (m+ n_e)- (m' +d+c) (m+n_1) = \sum_{j=2}^{e-1} b_j (n_e-n_j)$.
Since $\GCD\{n_e-n_1=s, \ldots , n_e-n_{e-1}\}=d$, such $c$ is a constant
depending only on $\{n_1,\ldots , n_e\}$ and $r$, which can take only
$s'$ different values.
Then we have $\al_1(m) = m'+ c$, since $ (m' +d+c) (m+n_1) =
(m'+c- \sum_{j=2}^{e-1} b_j ) (m+ n_e) + \sum_{j=2}^{e-1} b_j(m+n_j)$ and the minimality of $c$.
\end{proof}

Due to Lemma~\ref{homogeneous},  we write simply $\al_i= \al_i(m) $ for $m\gg 1$ and  $i \ne 1,e$.

\begin{Question} {\em If we assume $H=\langle n_1,n_2,n_3,n_4\rangle$ is almost symmetric
of type $3$, we have some examples of $d>1$ and odd, like
$H=\langle 20, 23, 44, 47\rangle$ with $d=3$ or $H= \langle 19, 24, 49, 54\rangle$
with $d=5$.  But in all examples we know, at least one of the minimal generators is even.
Is this true in general?  Note that we have examples of $4$ generated {\it symmetric}
semigroup all of whose minimal generators are odd.}
\end{Question}

\subsection{$H+m$ is almost symmetric of type $2$ for only finite $m$.}


\begin{Theorem}\label{H+m_type2}  Assume $H+m=\langle n_1+m,\ldots ,n_4+m\rangle$. Then for large enough   $m$,
$H+m$ is not almost symmetric of type $2$.
\end{Theorem}
\begin{proof}  We assume that $m$ is suitably big and  $H+m$ is almost symmetric  of type $2$.

Recall the RF-matrix is of the following form
(since we assumed the order on $\{n_1,\ldots, n_4\}$, we changed the indices
$\{1,2,3,4\}$ to $\{i,j,k,l\}$) by Theorem \ref{type2}.

\[ \RF(\Fr(H+m)/2) = \left( \begin{array}{cccc} -1 & \alpha_j(m)-1 & 0 & 0\\
0 & -1 & \alpha_k(m)-1 & 0 \\
\al_i(m)-1 & 0 & -1 & \al_l(m)-1\\
\al_i(m)-1 & \al_j(m)-1-\al_{ij}(m) & 0 & -1 \end{array}\right)\]

Also, we know by Lemma \ref{al_i(m)} that $\al_1(m), \al_4(m)$ grows linearly on $m$ and
$\al_2(m), \al_3(m)$ stays constant for big $m$.

Now, 1st and 2nd rows of   $\RF(\Fr(H+m)/2)$ shows $\Fr(H+m)/2 + (n_i+m)  = (\al_j(m)-1)(n_j +m)$
and $\Fr(H+m)/2 + (n_j+m) = (\al_k(m) - 1)(n_k+m)$.
Since $\Fr(H+m)$ grows by the order of $m^2$, these equations show that $\{j,k\} = \{1,4\}$.
But then looking at the 3rd row, since $\{i,l\} = \{2,3\}$, $\Fr(H+m)/2$ grows as a linear function
on $m$. A contradiction!
\end{proof}

\subsection{The classification of $H$ such that $H+m$ is almost symmetric of type 3 for infinitely many $m$}
Unlike the case of type $2$, there are infinite series of $H+m$,  which are almost symmetric of type $3$
for infinitely many $m$.
The following example was given by T.~Numata.

\begin{ex} If $H= \langle 10,11,13,14\rangle$, then $H+4m$ is AS of type $3$ for all integer $m\ge 0$.
\end{ex}

\begin{ex} {\em  For the following $H$, $H+m$ is almost symmetric with type $3$ if
\begin{enumerate}
\item[(1)] $H= \langle10,11,13,14\rangle$, $m$ is a multiple of $4$.
\item[(2)] $H = \langle 10, 13, 15, 18\rangle$, $m$ is a multiple of $8$.
\item[(3)] $H = \langle 14, 19, 21, 26\rangle$, $m$ is a multiple of $12$.
\item[(4)] $H= \langle18, 25, 27, 34\rangle$, $m$ is a multiple of $16$.
\end{enumerate}}
\end{ex}

 In the following, we determine the type of numerical semigroup $H$
generated by $4$ elements and $H+m$ is almost symmetric of type $3$ for infinitely many $m$.

\begin{defn} Let $H=\langle n_1,n_2,n_3 ,n_4\rangle$ with $n_1<n_2<n_3<n_4$ and
we assume that $H+m$ is almost symmetric of type $3$
with $\PF(H+m) = \{ f(m), f'(m), \Fr(H+m)\}$ with $f(m) < f'(m), f(m) + f'(m)= \Fr(H+m)$.
  We say some invariant $\sigma(m)$ (e.g. $\Fr(H+m), f(m), f'(m)$)
of $H+m$ is $O(m^2)$ (resp. $o(m)$) if there is some positive constant $c$ such that
$\sigma(m) \ge cm^2$ (resp. $\sigma(m) \le cm$) for all $m$.
\end{defn}

\begin{lem}  The invariants $\Fr(H+m)$ and $f'(m)$ are $O(m^2)$ and $f(m)$ is $o(m)$.
\end{lem}

\begin{proof}
By Lemma \ref{(al_i-1)n_i}, we have

\[ \phi(m)+ (n_k +m) = (\al_i (m)-1 )(n_i +m) \]

for every $i$ and for some $k$ and $\phi(m) \in \PF'(H+m)$. If $i=2,3$ (resp. $1,4$), then
$\al_i(m)=\al_i$ is constant (resp. grows linearly on $m$) and $\phi(m)$ is $O(m)$
 (resp. $O(m^2)$).
\end{proof}

\medskip
The following lemma is very important in our discussion.

 \begin{lem}\label{(al_i-1)n_i} Assume  $H+m=\langle n_1+m,n_2+m,n_3+m ,n_4+m\rangle$ is almost symmetric of type $3$ for sufficiently large $m$.
 Then for every $i$, there exists $k\ne i$ and $\phi(m)\in \PF'(H+m)$
 such that $(\al_i(m)-1)(n_i+m) = \phi(m)+ (m+n_k)$.
 \end{lem}
\begin{proof} We have seen in Lemma \ref{special row} (ii) that there are at least $4$ relations of
type $(\al_i(m)-1)(n_i+m) = \phi(m)+ (m+n_k)$.  If for some $i$, the relation
$(\al_i(m)-1)(n_i+m) = \phi(m)+ (m+n_k)$ does not exist, then for some $j\ne i$, there must exist
relations
\[(\al_i(m)-1)(m+n_j) = f(m)+ (m+n_k)= f'(m)+(m+n_l),\]
which is absurd since $f(m)= o(m)$ and $f'(m)=O(m^2)$.
\end{proof}

Now, we will start the classification of $H$, almost symmetric of type $3$ and
$H+m$ is   almost symmetric of type $3$ for infinitely many $m$.

\begin{prop}\label{AG3-fm}  Assume $H$ and $H+m$ are almost symmetric of type $3$
 for infinitely many $m$.
We use notation as above and we put $d=\GCD (n_2-n_1, n_3-n_2,n_4-n_3)$.
If $H+m$ is almost symmetric of type $3$ for sufficiently big $m$, then
 the following statements hold.
\begin{enumerate}
\item[{\em (i)}] We have $\al_2=\al_3$ and $\al_1(m)= \al_4(m)+1$. If  we put $a=\al_2=\al_3$ and $b= \al_1(m)=\al_4(m)+1$,
we have the following $\RF(f(m)), \RF(f'(m))$.
Note that $\RF(f(m))$ does not depend on $m$
if $H+m$ is almost symmetric.

\[\RF ( f(m) ) = \left( \begin{array}{cccc} -1 & a-1 & 0 & 0\\
1 & -1 & a-2 & 0 \\
0 & a-2 & -1 & 1\\
0  & 0 & a-1 &  -1 \end{array}\right),
\RF(f'(m))= \left( \begin{array}{cccc} -1 &  0 & 1 & b-d-2\\
0 & -1 &  0 & b-d-1\\
b-1 & 0 & -1 & 0\\
b-2  & 1 &  0 &  -1 \end{array}\right),\]
where we put  $b= \al_1(m)$  and then $\al_4(m)= b-d$.
\item[{\em (ii)}]  The integer  $a=\al_2=\al_3$ is odd and we have
$n_2= n_1+(a-2)d, n_3= n_1+ ad, n_4= n_1+ (2a-2)d$.
\end{enumerate}
\end{prop}

\begin{proof}  We divide our proof into several steps.

\medskip
(1) By Lemma \ref{(al_i-1)n_i}, we have the relations
\[f(m) + (m+n_k) = (\al_2-1)(m+n_2), \quad f(m) +  (m+n_l) = (\al_3-1)(m+n_3)\]
in $H+m$.  Taking the difference, we have
\[ n_l-n_k =  (\al_3-\al_2)m + (\al_3-1)n_3- (\al_2-1)n_2.\]
Hence we must have $\al_2=\al_3$, since $m$ is sufficiently larger than
$n_1,\ldots, n_4$.  Now, we will put $\al_2=\al_3=a$.

Then we will determine $n_k,n_l$.  Since we have $n_k< n_l$, there are $3$ possibilities;

\medskip
 $n_k=n_1$ and $n_l=n_4$, \quad  $n_k=n_1$ and $n_l=n_2$, \quad
 or \quad $n_k=n_3$ and $n_l=n_4$.
If we have $n_k =n_3$ and $f(m)+ n_3 = (a-1)(m+n_2)$, $f(m) + n_1<  (a-1)(m+n_2)$ and
there is no way to express $f(m)+n_1$ as an element of $H+m$.  Hence we have $n_k=n_1$
and in the same manner, we can show $n_l=n_4$. 
  Thus we have obtained
 \[f(m) + (m+n_1 )= (a-1)(m+n_2), \quad f(m) +  (m+n_4 )= (a-1)(m+n_3),\]
which are 1st and 4th rows of $\RF(f(m))$.

\medskip

Now, put $f(m) = (a-2)e_m$ ($e_m\in \QQ$). Then we have
\[(*) \quad (a-2)(e_m - (n_2+m))= n_2-n_1, \quad (a-2)((n_3+m)-e_m) = n_4-n_3.\]

Now, put $e_m = e+m, t= e - n_2$ and $u= n_3 - e$.  Then from (*) we have
\[(**)\quad  n_1= e- (a-1)t, n_2= e -t, n_3= e+u, n_4 =e + (a-1)u\]
with $n_3-n_2= t +u \in \Z$.

\medskip
(2) Noting $a=\al_2(m)$, we put
\begin{eqnarray}
a(m+n_2) = c_1(m+n_1)+c_3(m+n_3)+c_4(m+n_4),\\
a(m+n_3)= c'_1(m+n_1)+ c'_2(m+n_2)+c'_4(m+n_4)
\end{eqnarray}

Since $x_2^a - x_1^{c_1}x_3^{c_3}x_4^{c_4}$ is a homogeneous equation,
we have $a= c_1+c_3+c_4$.  Also, since $(a-1)(n_3-n_2)> n_2-n_1$, we have $c_1>1$ and
in the same manner, we get $c'_4>1$.

\medskip

(3) By Lemma \ref{(al_i-1)n_i}, there should be relations

\[f'(m) + (m+n_k) = (\al_1(m)-1)(m+n_1), \quad f'(m) + (m+n_l)= (\al_4(m)-1)(m+n_4)  \]
for some $k\ne 1$ and $l\ne 4$.
By Lemma \ref{symm}, we must have $k=3$ or $4$, $l= 1$ or $2$. We will show that
$k=3$ and $l=2$.

Assume, on the contrary,   $f'(m) + (m+n_4) = (\al_1(m)-1)(m+n_1)$ and
 put $f'(m) + (m+n_3)= p_1(m)(m +n_1)+ p_2(m)(m+n_2)$,
noting that $(3.4)$ entry of $\RF(f'(m))$ is $0$.  Then since
$f'(m) + (m+n_4) = (\al_1(m)-1)(m+n_1)> f'(m) + (m+n_3)= p_1(m)(m +n_1)+ p_2(m)(m+n_2)
\ge (p_1(m)+p_2(m))(m +n_1)$,

we have $\al_1(m)-1\ge p_1(m)+p_2(m)+1$.

Taking the difference of
\begin{eqnarray*}f'(m) + (m+n_4) = (\al_1(m)-1)(m+n_1),\\
f'(m) + (m+n_3)= p_1(m)(m +n_1)+ p_2(m)(m+n_2),
\end{eqnarray*}
we have
\[(***) \quad n_4-n_3= (\al_1(m) - 1-p_1(m) -p_2(m))(m+n_1) - p_2(m)(n_2-n_1).\]

Using the definition in Lemma \ref{homogeneous}, since $n_4-n_3, n_2-n_1$ are divisible by $d$
and $m +n_1$ is relatively prime with $d$, $(\al_1(m) - 1-p_1(m) -p_2(m))$ should be a multiple of
$d$ and thus we have $\al_1(m) - 1-p_1(m) -p_2(m)\ge d$.

Then we have $p_2(m)(n_2-n_1) \ge d (m+n_1) +C_1$, where $C_1$ is a constant
not depending on $m$. On the other hand, we have seen that $p_2(m) < \al_1(m)$
and $n_2-n_1 < s$ and $|\al_1(m) - dm/ s|\le C$, by Lemma \ref{homogeneous}.
 Then we have $d (m+n_1)+C_1\le p_2(m)(n_2-n_1) < p_2(m)s <\al_1(m) s
 \le dm$, which is a contradiction!

Thus we get
$(\al_1(m)-1)(m+n_1) = f'(m) + (m+n_3)$ and in the same manner,
$(\al_4(m)-1)(m+n_4) = f'(m) + (m+n_2)$.

\medskip

(4) Since we have seen $c_1,c'_4>1$ in (2), by Lemma \ref{symm}, we have
 $f'(m)+ (m+ n_1) = p_3(m)(m+n_3)+p_4(m)(m+n_4)$ and
$f'(m)+ (m+ n_4) =  q_1(m)(m+n_1)+q_2(m)(m+n_2)$.
 If, moreover, we assume $p_3(m)=0$, then we will have
$f'(m)+ (m+ n_1)= p_4(m)(m+n_4)$  and $f'(m) +(m+n_2) = (\al_ 4(m) -1)(m+n_4)$,
which will lead to $(n_2-n_1) = (\al_4(m)-1-p_4(m))(m+n_4)$, a contradiction!
Hence we have $p_3(m) >0$ and in the same manner, $q_2(m)>0$.
Again by Lemma \ref{symm}, we have $c_4=0= c'_1$.

\medskip
(5) Let us fix $\RF(f(m))$. Using (**), we compute
\begin{eqnarray}
m+n_1 +  (a-2) (m+n_3) - (f(m) + (m+n_2)) = (a-2) (u-t)\\
m+n_4 +  (a-2) (m+n_2) - (f(m) + (m+n_3))= (a-2) (u-t) .
\end{eqnarray}

This shows that if $u=t$, we get desired $\RF(f(m))$.

Now, if $u>t$, since  we have $f(m) + (m+n_3)= c'_2 (m+n_2) + (c_4'-1) (m+n_4)$
 with $c'_2 + c'_4=a, c'_4 \ge 2$, then
$c'_2 (m+n_2) + (c_4'-1) (m+n_4) \ge m+n_4 +  (a-2) (m+n_2) = (f(m) + (m+n_3)) + (a-2) (u-t) >0$,
a contradiction. Similarly, $t>u$ leads to a contradiction, too.  So we have $u=t$ and
we have proved the form of $\RF(f(m))$   in Proposition  \ref{AG3-fm}.

\medskip
(6) From $f'(m) + (m+n_2) = (\al_4(m)-1)(m+n_4)$, add $n_1-n_2=n_3-n_4$ to
both sides, we get  $f'(m) + (m+n_1) = (m+n_3) + (\al_4(m)-2)(m+n_4)$, likewise
 $f'(m) + (m+n_4) = (m+n_2) + (\al_1(m)-2)(m+n_1)$

\medskip
(7)  We will show that $\al_4(m)= \al_1(m)-d$.  By equations
\[f'(m)+(m+n_3) =(\al_1(m) -1)(m+n_1),\]   \[f'(m)+(m+n_2) =(\al_4(m) -1)(m+n_4),\]
taking the difference, we have
\[n_3-n_2 + (\al_1(m) - \al_4(m))(m+n_1) = (\al_4(m) - 1)s.\]
Since $s, n_3-n_2$ are divisible by $d$ and $\GCD(m+n_1,d)=1$,
$\al_1(m)-\al_4(m)$ is also divisible by $d$ and since $(\al_4(m)-1)s$
is bounded by $(d+C')m$ for some constant $C'$ by Lemma \ref{homogeneous},
we have $\al_1(m)-\al_4(m)=d$.

\medskip
(8) Finally, we will show that $a$ is odd and $t=u=d$.  Since $(a-2)t = n_2-n_1$ and $2t= n_3-n_2$
are integers, if  we show that $a$ is odd, then $t,u\in \Z$. Recall that
\[an_2 = 2 n_1+ (a-2) n_3\]
and $a$ is minimal so that $an_2$ is representable by $n_2,n_3,n_4$.  Then if
$a$ is even, then it will contradict the fact that $a$ is minimal.\par

By the definition of $d$ we see that $t=u=d$.
\end{proof}

\begin{Theorem}\label{H+mAG3}  Assume that
$H=\langle n_1, n_2, n_3, n_4\rangle$ with $n_1<n_2<n_3<n_4$
and we assume that
$H$ and $H+m$ are almost symmetric of type $3$ for infinitely many
$m$.  Then putting $d=\GCD (n_2-n_1,n_3-n_2,n_4-n_3)$, $a=\al_2,
b= \al_1$ and $\PF(H)= \{ f,f', \Fr(H)\}$, $H$  has the following
characterization.
\begin{enumerate}
\item[{\em (i)}] $a$ and $d$ are odd, $\GCD(a,d)=1$ and $b\ge d+2$.
\item[{\em (ii)}] $\RF(f)$ and $\RF(f')$ have the following form.
\[\RF(f) =  \left( \begin{array}{cccc} -1 & a-1 & 0 & 0\\
1 & -1 & a-2 & 0 \\
0 & a-2 & -1 & 1\\
0  & 0 & a-1 &  -1 \end{array}\right),
\RF(f')= \left( \begin{array}{cccc} -1 &  0 & 1 & b-d-2\\
0 & -1 &  0 & b-d-1\\
b-1 & 0 & -1 & 0\\
b-2  & 1 &  0 &  -1 \end{array}\right).\]
\item[{\em (iii)}]
$n_1= 2a+ (b-d-2)(2a-2), n_2= n_1+(a-2)d, n_3= n_1+ad,
n_4= 2a+ (b-2)(2a-2)$.
\item[{\em (iv)}] If we put $H(a,b; d) = \langle n_1,n_2,n_3,n_4\rangle$,
it is easy to see that $H(a,b+1; d) = H(a,b;d) + (2a-2)$.
Since $H(a,b;d)$ is almost symmetric of type $3$ for every
$a,d$ odd, $\GCD(a,d)=1$ and $b\ge d+2$,
$H(a,b;d)+m$ is  almost symmetric of type $3$ for infinitely many $m$.

\item[{\em (v)}]    $I_H = (xw-yz, y^a - x^2z^{a-2}, z^{a}-y^{a-2}w^2, xz^{a-1}- y^{a-1} w,
x^b -z^2w^{b-d-2}, w^{b-d}-x^{b-2}y^2, x^{b-1}y - zw^{b-d-1}))$.
\end{enumerate}
\end{Theorem}

\begin{proof}  We have determined $\RF(f), \RF(f')$  in Proposition \ref{AG3-fm}
and saw that $n_2= n_1 +(a-2)d,  n_3= n_1+ ad, n_4= n_1+(2a-2)d$.
Here we write $n_i$ (resp. $f.f'$) instead of $n_i+m$ (resp.
$f(m), f'(m)$).
Then from 1st and 3rd rows of
$\RF(f')$, we get
\[bn_1 = 2 (n_1+ad) + (b-d-2)(n_1+(2a-2)d)\]
hence $n_1= 2a+ (b-d-2)(2a-2)$.  It is easy to see that
putting
 $n_2= n_1+(a-2)d, n_3= n_1+ad,
n_4= n_1+(2a-2)d$,  we have $f = (a-1)n_2-n_1
= (a-2)(2a+(a-1)(2b-d-4))$ and $f' = (b-1)n_1-n_3=(b-2)n_1-ad$.

Also, since $n_1$ is even, $d$ should be odd to make a numerical semigroup.

We put $I' = (xw-yz, y^a - x^2z^{a-2}, z^{a}-y^{a-2}w^2, xz^{a-1}- y^{a-1} w,
x^b -z^2w^{b-d-2}, w^{b-d}-x^{b-2}y^2, x^{b-1}y - zw^{b-d-1})$.
We can easily see those equations come from
$\RF(f(m))$ or $\RF(f'(m))$.   Hence $I'\subset I_{H+m}$.
Now we will show that $I' = I_H$ by showing $\dim_k S/ (I', x)= \dim_k S/(I_H, x)
= n_1$.

Now, we see that
\[S/(I', x)\cong k[y,z,w]/( yz, y^{a}, z^{a}-y^{a-2}w^2,y^{a-1}w,
 z^2w^{b-d-2}, w^{b-d},  zw^{b-d-1}),\]
whose base over $k$ we can take

 \[A \cup B \cup C \cup \{y^{a-1}\}\cup \{zw^{b-d-2}\},\]
\noindent
where $A = \{y^iw^j \;|\; 0\le i\le a-3, 0\le j \le b-d-1\}$,

$B = \{z^iw^j\;|\; 1\le i \le a-1, 0\le j\le b-d-3\},
C=  \{ y^{a-2}w^j\;|0\le j\le b-d-1\}$.  Note that $y^{a-2}w^{b-d-1}
= z^aw^{b-d-3}$.

Hence $\dim_k S/(I', x) = \sharp A +\sharp B + \sharp C + 2
= (a-2)(b-d) + (a-1)(b-d-2) + (b-d) +2   = n_1 = S/(I_H, x )$.
Thus  we have show that $I'=I_H$.

Also, it is easy to see that $\Soc( S/(I_H,x))$ is generated by $y^{a-1}, zw^{b-d-2},
y^{a-2}w^{b-d-1}= z^{a} w^{b-d-3}$ and since $xw = yz$ in $k[H]$,
we have $ (y^{a-1})\cdot (zw^{b-d-2}) = x  (y^{a-2}w^{b-d-1})$ in $k[H]$,
which shows that $H(a,b)$ is almost symmetric.
\end{proof}


\begin{thebibliography}{NNW}

\bibitem[BDF]{BDF} V. Barucci, D. E. Dobbs, M. Fontana, \textit{Maximality properties in numerical
semigroups and applications to one-dimensional analytically irreducible
local domains}, Memoirs of the Amer. Math. Soc. {\bf 598} (1997).

\bibitem[BF]{BF} V. Barucci, R. Fr\"oberg, \textit{One-dimensional almost Gorenstein rings},
J. Algebra, \textbf{188} (1997), 418-442.
\bibitem[BFS]{BFS}  V. Barucci, R. Fr\" oberg, M. \c Sahin, On free resolutions of some semigroup rings,
J. Pure Appl. Alg., {\bf 218} (2014), 1107-1116.
\bibitem[DGM]{DGM} M. Delgado, P.A. Garc\'ia-S\'anchez, J. Morais, NumericalSgps -
a GAP package, 0.95, 2006, http://www.gap-sytem.org/Packages/numericalsgps.

\bibitem[E]{E}  K. Eto, \textit{Almost Gorenstein monomial curves in affine four space}. J. Alg.,
\textbf{488} (2017), 362 - 387.

\bibitem[FGH]{FGH} R. Fr\"oberg, C. Gottlieb, R. H\"aggkvist, \textit{On numerical semigroups}, Semigroup
 Forum \textbf{35} (1987), 63-83.
\bibitem[GTT]{GTT} S. Goto, R. Takahashi, N. Taniguchi, \textit{Almost Gorenstein rings
 - towards a theory of higher dimension-}, J. Pure and Applied Algebra, {\bf 219} (2015), 2666-2712.
\bibitem[GW]{GW} S. Goto, K. Watanabe, \textit{On graded rings}, J. Math. Soc. Japan {\bf 30} (1978), 172-213.
\bibitem[H]{H} J. Herzog, \textit{Generators and relations of abelian semigroups and semigroup rings}, Manuscripta Math. \textbf{3} (1970), 175-193.

\bibitem[HW16]{HW16} J. Herzog, K. Watanabe,
\textit{Almost symmetric numerical semigroups and almost Gorenstein semigroup rings, }
RIMS Kokyuroku, Kyoto University,  {\bf 2008} (2016), 107-128.

\bibitem[HW17]{HW17} J. Herzog, K. Watanabe,
\textit{Almost symmetric numerical semigroups and almost Gorenstein semigroup rings generated by 4 ekements, }  RIMS Kokyuroku Kyoto University, {\bf 2051} (2017), 120-132.
\bibitem[Ko]{Ko} J. Komeda, \textit{On the existence of weierstrass points with a certain semigroup generated by 4 elements}, Tsukuba J. Math {\bf 6} (1982), 237-270.
\bibitem[Ku]{Ku} E. Kunz, \textit{The value-semigroup of a one-dimensional Gorenstein rings},
Proc. Amer. Math. Soc. \textbf{25} (1970), 748-751.


\bibitem[Mo]{Mo} A. Moscariello, On the type of an almost Gorenstein monomial curve,
J. of Algebra,  \textbf{456} (2016), 266 -- 277.
\bibitem[N]{N} H. Nari, \textit{Symmetries on almost symmetric numerical semigroups}, Semigroup Forum,
 \textbf{86} (2013), 140 - 154.
\bibitem[NNW]{NNW} H. Nari, T. Numata, K.-i. Watanabe, \textit{Genus of numerical semigroups generated by three elements}, J. Algebra, \textbf{358} (2012), 67-73.
\bibitem[Nu]{Nu} T. Numata, \textit{Almost symmetric numerical semigroups generated by four elements}, Proceedings of the Institute of Natural Sciences, Nihon University \textbf{48} (2013), 197-207.

\bibitem[RG]{RG} J. C. Rosales, P. A. Garc\'ia-S\'anchez, \textit{Numerical semigroups}, Springer Developments in Mathematics, Volume \textbf{20}, (2009).
\bibitem[RG2]{RG2} J. C. Rosales, P. A. Garc\'ia-S\'anchez, \textit{Constructing almost symmetric numerical semigroups from irreducible numerical semigroups}, preprint.
\bibitem[W]{W} K. Watanabe, \textit{Some examples of one dimensional Gorenstein domains}, Nagoya Math. J. \textbf{49} (1973), 101-109.
\end{thebibliography}
\end{document}